\documentclass[11pt]{amsart}
\usepackage{tikz-cd}
\usepackage{verbatim}
\usepackage[colorinlistoftodos]{todonotes}
\usepackage{amssymb}
\usepackage{colonequals}
\usepackage{enumitem}
\usepackage{mathtools}
\usepackage[bookmarks,colorlinks,breaklinks]{hyperref}  
\usepackage[top=2cm, bottom=4.5cm, left=3cm, right=3cm]{geometry}
\hypersetup{linkcolor=blue,citecolor=blue,filecolor=blue,urlcolor=blue} 
\usepackage{tikz}\usetikzlibrary{cd,fit,matrix,arrows,decorations.pathmorphing}
\tikzset{commutative diagrams/.cd}
\usepackage{bm}
\usepackage{color}
\usepackage{amsfonts,amsmath,amsthm,amssymb,amscd, hyperref}

\numberwithin{equation}{subsection}

\newtheorem{theorem}{Theorem}[section]

\newtheorem{lemma}[theorem]{Lemma}
\newtheorem{proposition}[theorem]{Proposition}
\newtheorem{conjecture}[theorem]{Conjecture}

\theoremstyle{definition}

\newtheorem{definition-theorem}[theorem]{Definition-Theorem}
\newtheorem{example}[theorem]{Example}

\newtheorem{question}[theorem]{Question}

\newtheorem{remark}[theorem]{Remark}

\theoremstyle{remark}

\newtheorem*{remark*}{Remark}

\usepackage{enumitem}
\setenumerate[1]{label=\textup{(\alph*)}}
\setenumerate[2]{label=\textup{(\roman*)}}

\renewcommand{\l}{\lambda}

\newcommand{\OO}{\mathcal{O}}
\newcommand{\Ob}{\overline{\mathcal{O}}}
\newcommand{\trdeg}{{\rm trdeg}}
\newcommand{\Z}{{\mathbb Z}}

\newcommand{\N}{{\mathbb N}}

\newcommand{\Q}{{\mathbb Q}}
\newcommand{\C}{{\mathbb C}}

\newcommand{\Fp}{{\mathbb F}_p}
\newcommand{\Fq}{{\mathbb F}_q}
\newcommand{\Fpbar}{{\overline{\mathbb F}_p}}

\newcommand{\Kbar}{{\overline{K}}}
\newcommand{\Lbar}{{\overline{L}}}

\newcommand{\hhat}{\widehat{h}} %\hat is too narrow; avoid

\begin{document}
\title[Collision of orbits]{Collision of orbits for families of polynomials defined over fields of positive characteristic}

\author{Shamil Asgarli}
\address{Department of Mathematics and Computer Science, Santa Clara University,  500 El Camino Real, Santa Clara, CA 95053. Email: \texttt{sasgarli@scu.edu}}

\author{Dragos Ghioca}
\address{Department of Mathematics, University of British Columbia, Vancouver, BC V6T 1Z2. Email: \texttt{dghioca@math.ubc.ca}}


%\date{\today}

%\bibliographystyle{abbrv}

\begin{abstract}
Let $L$ be a field of positive characteristic $p$ with a fixed algebraic closure $\overline{L}$, and let $\alpha_1,\alpha_2,\beta\in L$. For an integer $d\ge 2$, we consider the family of polynomials $f_{\lambda}(z) := z^d+\lambda$, parameterized by $\lambda\in\overline{L}$. Define $C(\alpha_1,\alpha_2;\beta)$ to be the set of all $\lambda\in\overline{L}$ for which there exist $m,n\in\N$ such that $f_{\lambda}^m(\alpha_1)=f_{\lambda}^n(\alpha_2)=\beta$. In other words, $C(\alpha_1,\alpha_2;\beta)$ consists of all $\lambda\in\overline{L}$ with the property that the orbit of $\alpha_1$ collides with the orbit of $\alpha_2$ under the same polynomial $f_{\lambda}$ precisely at the point $\beta$. Assuming $\alpha_1,\alpha_2,\beta$ are not all contained in a finite subfield of $L$, we provide explicit necessary and sufficient conditions under which $C(\alpha_1,\alpha_2;\beta)$ is infinite. We also discuss the remaining case where $\alpha_1,\alpha_2,\beta\in \overline{\mathbb F}_p$ and provide ample computational data that suggest a somewhat surprising conjecture. Our problem fits into a long series of questions in the area of unlikely intersections in arithmetic dynamics, which have been primarily studied over fields of characteristic $0$. Working in characteristic $p$ adds significant difficulties, but also reveals the subtlety of our problem, especially when some of the points lie in a finite field or when $d$ is a power of $p$.
\end{abstract}

\maketitle

\section{Introduction}
\label{sec:intro}

%%%%%%%%%%%%%%%%%%%%%%%%%%%%%%%%%%%%%%%%%%%%%%%%%%%%%%%%%%%%%%%%%%%%%%%%%%%%%%%

\subsection{Notation}

Throughout this paper, we denote by $\N$ the set of all positive integers. For each field $K$, we denote by $\Kbar$ an algebraic closure of $K$; if $K_0$ is the prime subfield of $K$, we let $\overline{K_0}$ be its algebraic closure inside $\Kbar$. We mention standard definitions from algebraic dynamics. For any self-map $\Phi$ on a quasiprojective variety $X$ and for any $n\in\N$, we let $\Phi^n$ be the $n$-th compositional iterate of $\Phi$; by convention, $\Phi^0$ is the identity map. We define the \emph{(strict) forward orbit} of a point $\alpha\in X$ as the set of all points 
$\Phi^n(\alpha)$, for $n\ge 1$. 
Similarly, we define the \emph{(strict) backward orbit} of $\alpha$ as the set
$\Ob_\Phi(\alpha)$, which consists of all $\gamma\in X$ such that there exists $m\in\N$ with $\Phi^m(\gamma)=\alpha$. A point $\alpha$ is \emph{preperiodic} under the action of $\Phi$ if  there exist $0\le m<n$ such that $\Phi^m(\alpha)=\Phi^n(\alpha)$; if $m=0$, then $\alpha$ is \emph{periodic}.

%%%%%%%%%%%%%%%%%%%%%%%%%%%%%%%%%%%%%%%%%%%%%%%%%%%%%%%%%%%%%%%%%%%%%%%%%%%%%%%

\subsection{The unlikely intersection principle in arithmetic dynamics}

Several central questions in arithmetic geometry are rooted in the principle of unlikely intersections; for more details, we refer the reader to the excellent book of Zannier \cite{Umberto}. Over the past 30 years, there has been considerable research in algebraic dynamics on questions framed by this general principle; we provide a couple of prominent examples below.

The dynamical Mordell-Lang (DML) conjecture considers a quasiprojective variety $X$ with an endomorphism $\Phi$ (over a field $L$ of characteristic $0$), an irreducible curve $C\subset X$, and a point $\alpha\in X(L)$. The conjecture predicts that the \emph{unlikely} occurrence of infinitely many points in $\OO_\Phi(\alpha)\cap C$ \emph{must} be explained only by the fact that $C$ is periodic under the action of $\Phi$ (i.e., $\Phi^n(C)\subseteq C$ for some $n\in\N$). For more details on the DML conjecture and a survey of some of the partial results, we refer the reader to \cite{BGT-book}.

In a different direction, we describe the problem of \emph{simultaneously preperiodic points} (which was itself motivated by \cite{M-Z-1,M-Z-2}) for an algebraic family of polynomials; for a concise survey discussing this research question, see \cite{25}. Consider a family of polynomials $f_\l$ of degree $d\ge 2$, whose coefficients depend polynomially on $\l\in\C$, and two points $\alpha,\beta\in\C$. The \emph{unlikely} existence of infinitely many parameters $\l\in\C$ such that both $\alpha$ and $\beta$ are preperiodic for $f_\l$ can only occur if $\alpha$ and $\beta$ are \emph{dynamically related} with respect to the \emph{entire} family of polynomials $f_\l$ (see \cite[Theorem~1.3]{Matt-2}, which extends the previous results of \cite{Matt, GHT-ANT}). For a broader survey of recent work and new research directions on unlikely intersection questions in arithmetic dynamics, we refer the reader to \cite{survey}. 

The vast majority of the proven results and open problems in arithmetic dynamics concern the algebraic dynamical systems defined over fields of characteristic $0$ (see \cite{survey} and the references therein). It is only recently that several outstanding conjectures have been considered in positive characteristic; in each case, new features emerge for the dynamical systems in characteristic $p$. These new intricacies are due to the presence of the Frobenius map (similar to the isotriviality issues appearing in \cite{Tom}) and the existence of additive polynomials of degree greater than one (which lead to new unlikely intersection questions as in \cite{BM-1, BM-2}). Indeed, the DML conjecture (see \cite{Banach, JIMJ, TAMS, Alina, XY, Y}), the problem of simultaneous preperiodic points for an algebraic family of polynomials (see \cite{1}), and the Zariski Dense Orbit conjecture (see \cite{Alice-Tom} for its formulation over fields of characteristic $0$ and \cite{G-Sina, G-Sina-2} for the problem over fields of characteristic $p$) are all more subtle when studied in a positive characteristic setting.

Motivated by \cite{Matt}, the second author studied the following problem in \cite{1}. Given an integer $d\ge 2$,  a field $L$ of characteristic $p$, and points $\alpha_1,\alpha_2\in L$, we obtained (see \cite[Theorem~1.1]{1}) necessary and sufficient conditions for the existence of infinitely many $\l\in \Lbar$ such that both $\alpha_1$ and $\alpha_2$ are preperiodic for the polynomial $f_\l(z)\colonequals z^d+\l$ (see also \cite{22} for extensions of this result to more general families of polynomials).  In this article, we expand the problem studied in \cite{1} to the following setting. Working again with the family of polynomials $f_\l(z)\colonequals z^d+\l$ and given  \emph{two starting points} $\alpha_1,\alpha_2$, we study the conditions under which the strict forward orbits of these starting points contain a given \emph{target point} $\beta$. We refer to this setting as the \emph{colliding orbits problem}. A similar question of colliding orbits was previously studied in the context of Drinfeld modules (see \cite{G-JNT}).

%%%%%%%%%%%%%%%%%%%%%%%%%%%%%%%%%%%%%%%%%%%%%%%%%%%%%%%%%%%%%%%%%%%%%%%%%%%%%%%

\subsection{Our main result}
In the present paper, we prove the following statement.

\begin{theorem}
\label{thm:main_0}
Let $L$ be a field of characteristic $p>0$, let $\alpha_1,\alpha_2,\beta\in L$ and let $d\ge 2$ be an integer. Consider the family of polynomials $f_\l(z)=z^d+\l$, parameterized by $\l\in\Lbar$. Assume $\alpha_1,\alpha_2,\beta$ are not all contained in a finite subfield of $L$. Then  the set
\begin{equation}
\label{eq:0001}
C(\alpha_1,\alpha_2;\beta)\colonequals\left\{\l\in\Lbar\colon \text{ there exist $m,n\in\N$ such that }f_\l^m(\alpha_1)=f_\l^n(\alpha_2)=\beta\right\}
\end{equation}
is infinite if and only if exactly one of the following two conditions holds:
\begin{itemize}
\item[(A)] $\alpha_1^d=\alpha_2^d$; 
\item[(B)] $d=p^\ell$ for some $\ell\in\N$ and there exists a finite subfield $\Fq\subset L$ such that $\delta_1\colonequals\alpha_2-\alpha_1\in\Fq^\ast$ and $\delta_2\colonequals\beta-\alpha_1\in\Fq$. Furthermore, the system of two equations:
\begin{equation}
\label{eq:system_010}
\left\{\begin{array}{ccc}
\delta_1 & = & \sum_{i=0}^{s_1-1}\gamma^{p^{ik\ell}}\\
\\
\delta_2 & = & \sum_{i=0}^{s_2-1}\gamma^{p^{ik\ell}}\end{array}\right.
\end{equation}
has a solution  $(\gamma,k,s_1,s_2)\in\Fq^\ast\times\N\times\N\times\N$. 
\end{itemize}
\end{theorem}

\begin{remark}
\label{rem:not_symmetric}
At first glance, alternative~(B) in Theorem~\ref{thm:main_0} appears asymmetrical, as the roles of $\alpha_1$ and $\alpha_2$ apparently cannot be interchanged. However, as we will show in Lemma~\ref{lem:actually}, we can always reduce the problem to the system~\eqref{eq:system_010}, where the constants are defined relative to $\alpha_1$ as $\delta_1\colonequals \alpha_2-\alpha_1$ and $\delta_2\colonequals\beta-\alpha_1$. 

Furthermore, the condition $\alpha_1\ne\alpha_2$ in alternative~(B) (since $\delta_1\in\mathbb{F}_{q}^{\ast}$) ensures that alternatives (A) and (B) are mutually exclusive. For $d=p^{\ell}$, alternative (A) becomes $\alpha_1^{p^\ell}=\alpha_2^{p^\ell}$, which is equivalent to $\alpha_1=\alpha_2$ in characteristic $p$. 
\end{remark}

By definition, for each  $\l\in C(\alpha_1,\alpha_2;\beta)$, we have $\beta\in \OO_{f_\l}(\alpha_1)\cap\OO_{f_\l}(\alpha_2)$, i.e., the orbits of $\alpha_1$ and $\alpha_2$ under the action of $f_\l$ collide at the point $\beta$. The existence of infinitely many such parameters $\l$ is clearly an unlikely event, which we show occurs only when $\alpha_1$, $\alpha_2$, and $\beta$ are dynamically related. Indeed, condition~(A) in Theorem~\ref{thm:main_0} states that the orbits of $\alpha_1$ and $\alpha_2$ \emph{merge after one iteration for all parameters $\l$}, i.e., $\OO_{f_\l}(\alpha_1)=\OO_{f_\l}(\alpha_2)$.

Condition~(B) provides a more subtle dynamical relation between our three points when $d=p^\ell$. For any $\l\in\Lbar$, a suitable iterate of the polynomial $f_\l(z)=z^{p^\ell}+\l$ commutes with any given translation polynomial $T_\xi(z)\colonequals z+\xi$ for $\xi\in\Fpbar$. Specifically, if $\xi\in \mathbb{F}_{p^{r\ell}}$ then $f_\l^r\circ T_\xi =T_\xi \circ f_\l^r$ (see equation~\eqref{eq:84} for $f_\l^r$). Since $\alpha_1,\alpha_2,\beta$ differ by elements from $\Fpbar$ in condition~(B), this commutativity establishes that these three points are indeed dynamically related with respect to our entire family of polynomials $f_\l(z)$. The precise role of the system~\eqref{eq:system_010} is explained in Section~\ref{sec:ii} (see Proposition~\ref{prop:findings}). 

For the converse direction of Theorem~\ref{thm:main_0}, we prove slightly stronger statements that also cover the case where $\alpha_1,\alpha_2,\beta\in\Fpbar$ (see Theorems~\ref{thm:converse_ii}~and~\ref{thm:converse_i}). In Section~\ref{subsec:fpbar}, based on extensive numerical experiments, we propose Conjecture~\ref{conf:fpbar} to address the case where $\alpha_1,\alpha_2,\beta\in\Fpbar$ and $d$ is not a power of $p$. In Section~\ref{sec:strategy}, we also formulate a general Conjecture~\ref{conj:general} regarding colliding orbits for families of polynomials.

%%%%%%%%%%%%%%%%%%%%%%%%%%%%%%%%%%%%%%%%%%%%%%%%%%%%%%%%%%%%%%%%%%%%%%%%%%%%%%%

\subsection{Further connections for our results}

Another motivation for our Theorem~\ref{thm:main_0} arises from  \cite{GHT-PJM}. A special case of \cite[Theorem~1.1]{GHT-PJM} can be formulated for the Legendre family of elliptic curves $E_t$ given by the equation $y^2=x(x-1)(x-t)$ as we vary $t\in\overline{\Q}$. We denote by $[k]_t$ the multiplication-by-$k$ map (for any $k\in\Z$) on $E_t$. We also let $\mathcal{E}$ be the elliptic surface corresponding to the Legendre family and let $[k]$ be the corresponding multiplication-by-$k$ map on $\mathcal{E}$. Any section $\mathcal{P}$ on $\mathcal{E}$ corresponds to an algebraic family of points $P_t\in E_t$ (for all but finitely many $t\in\overline{\Q}$). Then \cite[Theorem~1.1]{GHT-PJM} asserts that for any $3$ sections $\mathcal{P},\mathcal{Q},\mathcal{R}$ on $\mathcal{E}$, if the set    
$$C\left(\mathcal{P},\mathcal{Q};\mathcal{R}\right) \colonequals\left\{t\in\overline{\Q}\colon \text{ there exist $m,n\in\Z$ such that $[m]_t(P_t)=[n]_t(Q_t)=R_t$}\right\}$$
is infinite, then  at least one of the following two conditions must hold:
\begin{itemize}
\item there exist $a,b\in\Z$, not both equal to $0$, such that $[a](\mathcal{P})=[b](\mathcal{Q})$;
\item there exists $c\in\Z$ such that either $[c](\mathcal{P})=\mathcal{R}$ or $[c](\mathcal{Q})=\mathcal{R}$.
\end{itemize}
In other words, if there exist infinitely many $t\in\overline{\Q}$ such that the cyclic groups generated by both $P_t$ and $Q_t$ contain $R_t$ (in $E_t$), then at least $2$ of the $3$ sections must be linearly dependent globally, on the elliptic surface. Once again, we encounter the principle of unlikely intersections: the existence of infinitely many $t\in\overline{\Q}$ for which $R_t$ is contained in the orbits of \emph{both} $P_t$ and $Q_t$ under the action of $\Z$ is explained \emph{only} by a global dynamical relation between their corresponding sections, $\mathcal{P}, \mathcal{Q}$, and $\mathcal{R}$.

\begin{remark}
\label{rem:GCD}
The motivation for \cite[Theorem~1.1]{GHT-PJM} itself comes from the work of Hsia and Tucker \cite{HT} on a dynamical analogue of the classical \emph{GCD}-problem. Given multiplicatively independent $a,b\in\N$, the classical \emph{GCD}-problem (solved by Bugeaud, Corvaja, and Zannier \cite{BCZ}) provides good bounds for $\gcd(a^n-1,b^n-1)$ as a function of $n\in\N$. Several extensions of the result from \cite{BCZ} were obtained, going beyond the classical setting and studying the question for function fields, including in positive characteristic (see \cite{A-R, C-Z-1, C-Z-2, GHT-NJM}). Our own work in Theorem~\ref{thm:main_0} can be viewed in a similar light, as it is essentially a question about the greatest common divisor of orbits generated by two algebraic families of polynomials (see Remark~\ref{rem:GCD_2} and Section~\ref{subsec:fpbar}).
\end{remark}

The hypothesis in Theorem~\ref{thm:main_0} that $\beta$ is contained in the forward orbits of both $\alpha_1$ and $\alpha_2$ (under the action of $f_\l$) can be restated as follows: $\alpha_1$ and $\alpha_2$ are contained in the backward orbit of $\beta$ (under the action of $f_\l$). The study of backward orbits and their associated arboreal Galois groups is a topic of great interest, which originated in the work of Odoni \cite{27}. For a thorough overview, we refer the reader to the excellent survey by Jones \cite{26}; see also \cite{Rob-arb, Rob-arb-2, arb-3, 21, Tom, 23} for a sample of further work in this area. It is also worth noting that the case of a polynomial whose critical points have colliding orbits represents a special case in the study of the corresponding arboreal Galois groups (see \cite{23, 24, Rob-2, Rob-3}).

%%%%%%%%%%%%%%%%%%%%%%%%%%%%%%%%%%%%%%%%%%%%%%%%%%%%%%%%%%%%%%%%%%%%%%%%%%%%%%%

\subsection{Plan for our paper}

In Section~\ref{sec:strategy}, we present the general strategy for our proof of Theorem~\ref{thm:main_0}. We split the content of Theorem~\ref{thm:main_0} into three distinct results: Theorems~\ref{thm:main}, \ref{thm:converse_ii}, and ~\ref{thm:converse_i}. In Section~\ref{sec:strategy}, we also outline future research directions in the area of colliding orbits by formulating Conjecture~\ref{conj:general} for arbitrary families of polynomials. For a brief discussion of the characteristic $0$ case, see Remark~\ref{subsec:char_0}.

In Sections~\ref{sec:preliminary},~\ref{sec:heights}~and~\ref{sec:bounds}, we establish useful preliminary results, which are then employed in the proof of Theorem~\ref{thm:main}. We finish its proof in Section~\ref{sec:proof}, thus completing the direct implication in Theorem~\ref{thm:main_0}.

In Section~\ref{sec:i}, we complete the proof of Theorem~\ref{thm:converse_i}, while in Section~\ref{sec:ii}, we prove Theorem~\ref{thm:converse_ii}. Combined, these two results provide the converse implication in Theorem~\ref{thm:main_0}.

We conclude by discussing in Section~\ref{subsec:fpbar} the case of colliding orbits when the starting points $\alpha_1,\alpha_2$ and the target point $\beta$ all live in a finite field. We believe (see Conjecture~\ref{conf:fpbar}) that in this case, the set $C(\alpha_1,\alpha_2;\beta)$ is infinite, provided $d$ is not a power of $p$. We have ample numerical evidence to support this conjecture; furthermore, we formulate additional questions and conjectures all predicting a higher-than-expected frequency of unlikely intersections when the entire dynamical system is defined over $\Fpbar$. 

\bigskip

\textbf{Acknowledgments.} The second author was partially supported by an NSERC Discovery grant. We warmly thank the referee for their helpful comments and suggestions. 

%%%%%%%%%%%%%%%%%%%%%%%%%%%%%%%%%%%%%%%%%%%%%%%%%%%%%%%%%%%%%%%%%%%%%%%%%%%%%%%
%%%%%%%%%%%%%%%%%%%%%%%%%%%%%%%%%%%%%%%%%%%%%%%%%%%%%%%%%%%%%%%%%%%%%%%%%%%%%%%

\section{Strategy for our proof of Theorem~\ref{thm:main_0} and further extensions}
\label{sec:strategy}

In Subsection~\ref{subsec:direct}, we state Theorem~\ref{thm:main} and then explain its proof strategy. In Subsection~\ref{subsec:converse}, we state Theorems~\ref{thm:converse_ii}~and~\ref{thm:converse_i} and briefly mention key ideas in their proofs. We discuss possible extensions of our results in Subsection~\ref{subsec:conjecture}.

%%%%%%%%%%%%%%%%%%%%%%%%%%%%%%%%%%%%%%%%%%%%%%%%%%%%%%%%%%%%%%%%%%%%%%%%%%%%%%%

\subsection{The direct implication in Theorem~\ref{thm:main_0}}
\label{subsec:direct}
We will prove the following.

\begin{theorem}
\label{thm:main}
Let $L$ be a field of prime characteristic $p$, let $\alpha_1,\alpha_2,\beta\in L$ and let $d\ge 2$ be an integer. Consider the family of polynomials $f_\l(z)=z^d+\l$, parameterized by $\l\in\Lbar$. Let $\Fpbar$ denote the algebraic closure of $\Fp$ inside $\Lbar$. If the set
\begin{equation}
\label{eq:1}
C(\alpha_1,\alpha_2;\beta)\colonequals\left\{\l\in\Lbar\colon \text{ there exist $m,n\in\N$ such that }f_\l^m(\alpha_1)=f_\l^n(\alpha_2)=\beta\right\}
\end{equation}  
is infinite, then at least one of the following conditions must hold:
\begin{itemize}
\item[(i)] $\alpha_1^d=\alpha_2^d$.
\item[(ii)] $d=p^\ell$ for some $\ell\in\N$, and $\alpha_1-\beta, \alpha_2-\beta\in\Fpbar$.
\item[(iii)] $\alpha_1,\alpha_2,\beta\in \Fpbar$.
\end{itemize}
\end{theorem}

We now outline the strategy for proving Theorem~\ref{thm:main}. First, assuming condition~(iii) does \emph{not} hold, the infinitude of the set $C(\alpha_1,\alpha_2;\beta)$ implies that either condition~(i) is satisfied, or $\trdeg_{\Fp}(\Fp(\alpha_1,\alpha_2,\beta))=1$; this reduction is proved in Subsection~\ref{subsec:trdeg} through a series of Lemmas and Propositions. This allows us to set up the height machine in Section~\ref{sec:heights}; in particular, we obtain a key technical statement (see Proposition~\ref{prop:useful_variation}) regarding the variation of the (global) canonical height $\hhat_\l(\alpha)$ of a point $\alpha$ (associated to a polynomial $f_\l$ from our family) compared to the Weil heights of $\l$ and $\alpha$. In turn, Proposition~\ref{prop:useful_variation} allows us to show (see Proposition~\ref{prop:useful}) that there exists a sequence $\{\l_k\}_{k\in\N}\subseteq  C(\alpha_1,\alpha_2;\beta)$ such that
\begin{equation}
\label{eq:to_0}
\lim_{k\to\infty} \hhat_{\l_k}(\alpha_1)= \lim_{k\to\infty}\hhat_{\l_k}(\alpha_2)=0.
\end{equation}
Equation~\eqref{eq:to_0} is the crucial hypothesis needed to apply Theorem~\ref{thm:4.1}, which leads to an equality of the canonical heights of $\alpha_1$ and $\alpha_2$ with respect to \emph{each} polynomial $f_\l$. Then Theorem~\ref{thm:5.1} immediately provides the desired conclusion in Theorem~\ref{thm:main} when $d$ is not a power of $p$. The remaining case, where $d=p^\ell$ for some $\ell\in\N$, requires a more in-depth analysis to arrive at condition~(ii) in Theorem~\ref{thm:main} (see Propositions~\ref{prop:diff_fpbar}~and~\ref{prop:diff_fpbar_2}). To obtain the more precise information from condition~(B) in Theorem~\ref{thm:main_0} regarding system~\eqref{eq:system_010}, we will rely on Theorem~\ref{thm:converse_ii}.

%%%%%%%%%%%%%%%%%%%%%%%%%%%%%%%%%%%%%%%%%%%%%%%%%%%%%%%%%%%%%%%%%%%%%%%%%%%%%%%

\subsection{The converse implication in Theorem~\ref{thm:main_0}}
\label{subsec:converse}
In Section~\ref{sec:ii}, we prove the following result, which provides (along with Theorem~\ref{thm:main}) the full conclusion in Theorem~\ref{thm:main_0} when $d=p^\ell$.

\begin{theorem}
\label{thm:converse_ii}
Let $L$ be a field of characteristic $p$, let $\alpha_1,\alpha_2,\beta\in L$ with $\alpha_1\ne\alpha_2$, and let $d=p^\ell$ for some $\ell\in\N$. For each $\l\in\Lbar$, we let $f_\l(z)=z^d+\l$. Consider the set
$$C(\alpha_1,\alpha_2;\beta)=\left\{\l\in\Lbar\colon\text{ there exist $m,n\in\N$ such that $f_\l^m(\alpha_1)=f_\l^n(\alpha_2)=\beta$}\right\}$$
and let $\delta_1\colonequals \alpha_2-\alpha_1$ and $\delta_2 \colonequals \beta-\alpha_1$. Assume there exists a finite subfield $\Fq\subseteq L$ such that 
\begin{equation}
\label{eq:790}
\delta_1\in\Fq^\ast\text{ and }\delta_2\in\Fq. 
\end{equation}
Then the set $C(\alpha_1,\alpha_2;\beta)$ is infinite if  the system of two equations:
\begin{equation}
\label{eq:system_10}
\left\{\begin{array}{ccc}
\delta_1 & = & \sum_{i=0}^{s_1-1}\gamma^{p^{ik\ell}}\\
\\
\delta_2 & = & \sum_{i=0}^{s_2-1}\gamma^{p^{ik\ell}}\end{array}\right.
\end{equation}
has a solution  $(\gamma,k,s_1,s_2)\in\Fq^\ast\times\N\times\N\times\N$. Moreover, the set $C(\alpha_1,\alpha_2;\beta)$ is empty if the system~\eqref{eq:system_10} has no solution $(\gamma,k,s_1,s_2)\in\Fq^\ast\times\N\times \N\times\N$.
\end{theorem}

\begin{remark}
It is interesting that under the hypotheses~\eqref{eq:790} of Theorem~\ref{thm:converse_ii}, either $C(\alpha_1,\alpha_2;\beta)$ is infinite or empty. Also, in this case, for \emph{each} $\l\in C(\alpha_1,\alpha_2;\beta)$, we have that $\alpha_1,\alpha_2,\beta$ are \emph{all} preperiodic under the action of $f_\l$ (see Lemmas~\ref{lem:prep_0}~and~\ref{lem:prep}).

On the other hand, under the assumption that $d=p^\ell$, it could be that $C(\alpha_1,\alpha_2;\beta)$ is finite and nonempty, but this can only occur if either $\alpha_1-\alpha_2$ or $\alpha_1-\beta$ is \emph{transcendental} over $\Fp$.
\end{remark}

The following result shows that condition~(A) in Theorem~\ref{thm:main_0} implies the existence of infinitely many parameters $\l\in C(\alpha_1,\alpha_2;\beta)$ (see~\eqref{eq:0001}).

\begin{theorem}
\label{thm:converse_i}
Let $L$ be a field of characteristic $p$, let $\alpha,\beta\in L$, let $d\ge 2$ be an integer, and let $f_\l(z)=z^d+\l$ be a family of polynomials parameterized by $\l\in\Lbar$. Then there exist infinitely many $\l\in\Lbar$ such that  $f_\l^m(\alpha)=\beta$ for some $m\in\mathbb{N}$ (where $m$ depends on $\lambda$).
\end{theorem}

Indeed, under the condition~(A) from Theorem~\ref{thm:main_0} that $\alpha_1^d=\alpha_2^d$, we have that $f_\l^n(\alpha_1)=f_\l^n(\alpha_2)$ for \emph{all} $n\in\N$. Applying Theorem~\ref{thm:converse_i} to starting point $\alpha_1$ and target point $\beta$, we obtain that the set $C(\alpha_1,\alpha_2;\beta)$ from equation~\eqref{eq:0001} must be infinite. 

In Section~\ref{sec:i}, we prove Theorem~\ref{thm:converse_i} by contradiction. Using the assumption that the equations in $\l$ of the form $f_\l^m(\alpha)=\beta$ (as we vary $m$) have only finitely many solutions, we obtain that a certain plane curve (see equation~\eqref{eq:105}) contains \emph{infinitely} many points from a suitable finitely generated subgroup of $\mathbb{G}_m^2$. This allows us to apply the main result of Moosa-Scanlon \cite{F} on the structure of the intersection between a subvariety of a torus (in characteristic $p$) with a finitely generated subgroup to derive a contradiction.

\begin{remark}
\label{rem:fpbar}
We emphasize that both Theorems~\ref{thm:converse_i}~and~\ref{thm:converse_ii} also hold when $\alpha_1,\alpha_2,\beta\in\Fpbar$, i.e., the converse implication in Theorem~\ref{thm:main_0} does not depend on whether all three points $\alpha_1,\alpha_2,\beta$ are in $\Fpbar$.
\end{remark}

%%%%%%%%%%%%%%%%%%%%%%%%%%%%%%%%%%%%%%%%%%%%%%%%%%%%%%%%%%%%%%%%%%%%%%%%%%%%%%%

\subsection{A general conjecture}
\label{subsec:conjecture}

It is natural to consider a general colliding orbits problem for arbitrary families of polynomials in normal form. A polynomial of degree $d\ge 2$ is in \emph{normal form} if it is monic and its coefficient for the monomial $x^{d-1}$ is $0$. In Conjecture~\ref{conj:general}, we allow both the starting points $\alpha_1,\alpha_2$ and the target point $\beta$ to vary in an algebraic family as well.

\begin{conjecture}
\label{conj:general}
Let $L$ be a field of characteristic $p$, and let $\alpha_1(z),\alpha_2(z),\beta(z)\in L[z]$. Suppose $f_\l(x)\in L[x]$ is a family of polynomials of degree $d\ge 2$ (parameterized by $\l\in \Lbar$) in normal form, i.e.
\begin{equation}
\label{eq:conj_form}
f_\l(x)=x^d + \sum_{i=0}^{d-2}c_i(\l)\cdot x^i,
\end{equation}
for some polynomials $c_i(z)\in L[z]$ for $i=0,\dots, d-2$. We let 
$$C(\alpha_1,\alpha_2;\beta)=\left\{\l\in\Lbar\colon\text{ $f_\l^m(\alpha_1(\l))=\beta(\l)$ and $f_\l^n(\alpha_2(\l))=\beta(\l)$ for some $m,n\in\N$}\right\}.$$ 
If $C(\alpha_1,\alpha_2;\beta)$ is infinite, then at least one of the following conditions must hold:
\begin{enumerate}[label=(\arabic*)]
\item\label{conj:item-1} there exists a family of polynomials $g_\l(x)$ (similar to \eqref{eq:conj_form}, but not necessarily normalized) and there exist integers $k>0$ and $m,n\ge 0$ such that
\begin{equation} 
\label{eq:conj_1}
f_\l^k\circ g_\l=g_\l\circ f_\l^k\text{ and }f_\l^m(\alpha_1(\l))=g_\l\left(f_\l^n(\alpha_2(\l))\right)\text{ (or $f_\l^m(\alpha_2(\l))=g_\l\left(f_\l^n(\alpha_1(\l))\right)$),}
\end{equation}
for all $\l\in \Lbar$.
\item\label{conj:item-2} there exists $k\in\N$ such that for some $j\in\{1,2\}$, we have that $f_\l^k(\alpha_j(\l))=\beta(\l)$ for \emph{all} $\l\in\Lbar$. 
\item\label{conj:item-3} for \emph{each} $\l\in \Lbar$, the polynomial $\tilde{f}_\l(x)\colonequals f_\l(x)-c_0(\l)$ is additive (i.e., $\tilde{f}_\l(x+y)=\tilde{f}_\l(x)+\tilde{f}_\l(y)$ for all $x,y$). Furthermore,  for each $\l\in\Lbar$, we have that $\delta_1(\l)\colonequals\alpha_2(\l)-\alpha_1(\l)$ and $\delta_2(\l)\colonequals\beta(\l)-\alpha_1(\l)$ are preperiodic under the action of $\tilde{f}_\l(x)$.
\item\label{conj:item-4} $c_i(z)\in\Fpbar[z]$ for $i=0,\dots, d-2$ and $\alpha_1(z),\alpha_2(z),\beta(z)\in\Fpbar[z]$.
\end{enumerate}
\end{conjecture}

\begin{remark}
\label{rem:simplified_form}
One could formulate Conjecture~\ref{conj:general} for an arbitrary (unnormalized) family of polynomials $f_\l\in L[\l][x]$ (of degree $d\ge 2$), but this would complicate condition~\ref{conj:item-4}. When $p\nmid d$, this simplification has no loss of generality: any such family can be normalized through a linear conjugation, at the expense of parameterizing $\l$  by a curve rather than the affine line. 
\end{remark}

We briefly discuss the statements~\ref{conj:item-1}-\ref{conj:item-4} in the conclusion of Conjecture~\ref{conj:general}. Statement~\ref{conj:item-1} is a significant generalization of conclusion~(A) in Theorem~\ref{thm:main_0}; for arbitrary families of polynomials~\eqref{eq:conj_form}, the starting points $\alpha_1(\l)$ and $\alpha_2(\l)$ may be dynamically related through the much more complicated relation~\eqref{eq:conj_1} from Conjecture~\ref{conj:general} (see also \cite[Theorem~1.3]{Matt-2}). It is likely that to obtain a converse statement in Conjecture~\ref{conj:general}, i.e., that the set $C(\alpha_1,\alpha_2;\beta)$ is infinite, one would need a stronger statement than~\ref{conj:item-1}.

Statement~\ref{conj:item-2} does not appear for the dynamical system considered in Theorem~\ref{thm:main_0}. However, for general dynamical systems, one needs to account for the possibility that $\beta(\l)$ is in the forward orbit of $\alpha_1(\l)$ or $\alpha_2(\l)$ (for \emph{all} $\l\in\Lbar$). If this were to happen, one would expect the corresponding set $C(\alpha_1,\alpha_2;\beta)$ to be infinite due to a possible extension of our Theorem~\ref{thm:converse_i}. 

Statement~\ref{conj:item-3} asks that the only monomials $x^i$ in $f_\l(x)$ (for $i>0$) appearing with a nonzero coefficient correspond to $i=p^j$ for some $j\ge 0$; this is the generalization of conclusion~(B) appearing in Theorem~\ref{thm:main_0}. Again, a converse to Conjecture~\ref{conj:general} would require a more refined version of statement~\ref{conj:item-3} (see the system~\eqref{eq:system_010} from Theorem~\ref{thm:main_0} for $f_\l(z)=z^d+\l$).

Finally, we expect statement~\ref{conj:item-4} from Conjecture~\ref{conj:general} yields that  $C(\alpha_1,\alpha_2;\beta)$ is infinite as long as the family $f_\l(z)$ does \emph{not} satisfy statement~\ref{conj:item-3} (similar to Conjecture~\ref{conf:fpbar} for the special family of polynomials $f_\l(z)=z^d+\l$) \emph{and} also assuming that neither $\alpha_1$ nor $\alpha_2$ is \emph{persistently preperiodic} for our family of polynomials (i.e., for $j=1,2$,  there exists no $0\le m<n$ such that $f_\l^m(\alpha_j(\l))=f_\l^n(\alpha_j(\l))$ for \emph{all} $\l\in\Lbar$). In Subsection~\ref{subsec:9.4}, we present some numerical evidence supporting our expectation.

We expect Conjecture~\ref{conj:general} to be \emph{very difficult}. The main obstacle is the lack of a generalization of Theorem~\ref{thm:5.1} (obtained in \cite{1} for $f_\l(z)\colonequals z^d+\l$) to arbitrary families of polynomials (so far, \cite{22} provides the most general known extension of the main result from \cite{1} to families of polynomials of degrees coprime to the characteristic of the given field). This difficulty, in turn, stems from the fact that some of the key tools available in characteristic $0$ for determining the \emph{exact dynamical relation} between points of equal canonical height (with respect to the given family of polynomials) do not exist in characteristic $p$. 

\begin{remark}
\label{subsec:char_0}
We expect that the analogue of Conjecture~\ref{conj:general} for dynamical systems over a field $L$ of characteristic $0$ is more manageable and would only lead to the conclusions~\ref{conj:item-1}~and~\ref{conj:item-2}. The general strategy would follow the steps from our paper. While there would be no complications from additive polynomials (which are nontrivial only in characteristic $p$), new technical challenges would arise from dealing with arbitrary polynomial families. A potential line of attack would be to use the variation of the canonical height in families, as proven by Ingram \cite{Patrick}, to generalize \eqref{eq:to_0} when $C(\alpha_1,\alpha_2;\beta)$ is infinite. Then using \cite[Theorem~1.3]{Matt-2}, together with the description of the periodic plane curves under the coordinatewise action of polynomials from \cite{Alice-Tom}, one should be able to derive the conclusions~\ref{conj:item-1}-\ref{conj:item-2} from Conjecture~\ref{conj:general}. There are further complications when the dynamical system $(f_\l,\alpha_1(\l),\alpha_2(\l),\beta(\l))$ is not defined over $\overline{\Q}$. However, we believe these can be overcome using the methods from \cite[Section~10]{GHT-ANT} and the description of points of canonical height $0$ from \cite{Rob} for a polynomial defined over a function field. We hope to return to this general question for dynamical systems over fields of characteristic $0$ in a sequel paper.
\end{remark}

%%%%%%%%%%%%%%%%%%%%%%%%%%%%%%%%%%%%%%%%%%%%%%%%%%%%%%%%%%%%%%%%%%%%%%%%%%%%%%%
%%%%%%%%%%%%%%%%%%%%%%%%%%%%%%%%%%%%%%%%%%%%%%%%%%%%%%%%%%%%%%%%%%%%%%%%%%%%%%%

\section{Preliminary results}\label{sec:preliminary}

Let $L$ be a field of characteristic $p$ with a fixed algebraic closure $\Lbar$. We consider the family of polynomials $f_\l(z)=z^d+\l$ parameterized by $\l\in\Lbar$.

This Section is divided into two parts. In Subsection~\ref{subsec:1}, we gather some information about the iterates of $f_\l(z)$. In Subsection~\ref{subsec:trdeg}, we work under the hypotheses of Theorem~\ref{thm:main} for the points $\alpha_1,\alpha_2,\beta\in L$ (i.e., assuming the set $C(\alpha_1,\alpha_2;\beta)$ from~\eqref{eq:1} is infinite) to obtain information regarding the transcendence degree of $\Fp(\alpha_1,\alpha_2,\beta)/\Fp$.

%%%%%%%%%%%%%%%%%%%%%%%%%%%%%%%%%%%%%%%%%%%%%%%%%%%%%%%%%%%%%%%%%%%%%%%%%%%%%%%

\subsection{Iterates of our family of polynomials}
\label{subsec:1}

Let $\alpha\in L$. Following \cite[Subsection~3.1]{1}, for each $n\in\N$, there exists a polynomial $P_{n,\alpha}(\l)\in L[\l]$ of degree $d^{n-1}$ such that
\begin{equation}
\label{eq:2}
P_{n,\alpha}(\l)=f_\l^n(\alpha)\text{ for each }\l\in\Lbar.
\end{equation} 

\begin{remark}
\label{rem:GCD_2}
Using the notation from \eqref{eq:2}, we can reformulate the hypotheses from Theorem~\ref{thm:main_0} (or from Theorem~\ref{thm:main}) that $C(\alpha_1,\alpha_2;\beta)$ is infinite as follows. Given $\alpha_1,\alpha_2,\beta\in L$, we are asking that there exist infinitely many $\l\in\Lbar$ such that for some $m,n\in\N$, $$P_{m,\alpha_1}(\l)-\beta=P_{n,\alpha_2}(\l)-\beta=0;$$
or alternatively, the set of roots for all  polynomials $\gcd\left(P_{m,\alpha_1}(x)-\beta,P_{n,\alpha_2}(x)-\beta\right)\in L[x]$ (as we vary $m,n\in\N$) is infinite. 
\end{remark}

A simple induction on $n$ yields the following result.
\begin{lemma}
\label{lem:poly}
For each $n\in\mathbb{N}$, the coefficients of the polynomial $P_{n,\alpha}(\lambda)$ are themselves polynomials in $\alpha$, that is, 
\begin{equation}
\label{eq:3}
P_{n,\alpha}(\l)=\sum_{i=0}^{d^{n-1}} c_{n,i}(\alpha)\l^{d^{n-1}-i}.
\end{equation}
Furthermore, we have the following more precise information:
\begin{itemize}
\item[(a)] $c_{n,0}(\alpha)=1$;
\item[(b)] $c_{n,d^{n-1}}(\alpha)=\alpha^{d^n}$;
\item[(c)] $\deg_\alpha(c_{n,i}(\alpha))\le d\cdot i$, for $i=0,\dots, d^{n-1}$.   
\end{itemize}
\end{lemma}

\begin{proof}
We prove that statements (a)-(c) hold by induction on $n$; the case $n=1$ is obvious since $f_\l(\alpha)=\alpha^d+\l$ and so, $c_{1,0}(\alpha)=1$, while $c_{1,1}(\alpha)=\alpha^d$.

Now, we assume the statements (a)-(c) hold for $c_{n,i}$ (for $0\le i\le d^{n-1}$) and we prove the same statements also hold for $c_{n+1,i}$ and $0\le i\le d^n$. We have 
\begin{equation}
\label{eq:0}
P_{n+1,\alpha}(\l)= P_{n,\alpha}(\l)^d + \l= \left(\sum_{i=0}^{d^{n-1}}c_{n,i}(\alpha) \l^{d^{n-1}-i}\right)^d+\l.
\end{equation}
By inspecting the expansion, $c_{n+1,0}(\alpha)=c_{n,0}(\alpha)^d$ and $c_{n+1,d^n}(\alpha)= c_{n,d^{n-1}}(\alpha)^d$. Hence, statements~(a)-(b) follow by the inductive hypothesis.  

Finally, regarding statement~(c), we assign weight $d$ to $\l$ and weight $1$ to $\alpha$. The total weight of each monomial (in $\alpha$ and $\l$) from $P_{n,\alpha}(\l)$ is therefore at most $d^n$ (using the inductive hypothesis for~(c) along with equation~\eqref{eq:3}). Then using the recurrence relation~\eqref{eq:0}, we conclude that each monomial in $P_{n+1,\alpha}(\l)$ has weight at most $d^{n+1}$; therefore, the degree of each $c_{n+1,i}(\alpha)$ is at most $d\cdot i$, as desired. \end{proof}

The following result is an easy consequence of Lemma~\ref{lem:poly}.
\begin{lemma}
\label{lem:finitely_l}
For each $\alpha,\beta\in L$ and for each $n\in\N$, there exist finitely many $\l\in\Lbar$ such that $f_\l^n(\alpha)=\beta$.
\end{lemma}

\begin{proof}
Since $P_{n,\alpha}(\l)=\beta$ (see~\eqref{eq:3}) is an equation of degree $d^{n-1}$ (in $\l$), there are only finitely many solutions $\lambda\in \Lbar$.
\end{proof}

%%%%%%%%%%%%%%%%%%%%%%%%%%%%%%%%%%%%%%%%%%%%%%%%%%%%%%%%%%%%%%%%%%%%%%%%%%%%%%%

\subsection{Transcendence degree for the field generated by our points}
\label{subsec:trdeg}

We let $\alpha_1,\alpha_2,\beta\in L$ and let $C(\alpha_1,\alpha_2;\beta)$ be the subset of $\Lbar$ (as in \eqref{eq:1}) consisting of all $\l$ for which there exist some $m,n\in\N$ such that
\begin{equation}
\label{eq:4}
f_\l^m(\alpha_1)=f_\l^n(\alpha_2)=\beta.
\end{equation}

\begin{lemma}
\label{lem:trdeg_lambda}
If $\l\in C(\alpha_1,\alpha_2;\beta)$, then $\l\in\overline{\Fp(\alpha_1,\beta)}\bigcap \overline{\Fp(\alpha_2,\beta)}.$
\end{lemma}

\begin{proof}
Writing $f_\l^m(\alpha_1)=\beta$ as $P_{m,\alpha_1}(\l)=\beta$ (see Lemma~\ref{lem:poly}), we obtain the following equation for $\l$:
\begin{equation}
\label{2-eq}
\l^{d^{m-1}}+\sum_{i=0}^{d^{m-1}}c_{m,i}(\alpha_1)\l^{d^{m-1}-i}=\beta.
\end{equation}
In particular, equation~\eqref{2-eq} implies that $\l\in \overline{\Fp(\alpha_1,\beta)}$. By symmetry, we also have $\l\in\overline{\Fp(\alpha_2,\beta)}$.
\end{proof}

\begin{lemma}
\label{lem:trdeg_not_3}
If $C(\alpha_1,\alpha_2;\beta)$ is nonempty, then $\alpha_1\in \overline{\Fp(\alpha_2,\beta)}$ and $\alpha_2\in \overline{\Fp(\alpha_1,\beta)}$. 
\end{lemma}

\begin{proof}
Let $\l\in C(\alpha_1,\alpha_2;\beta)$, i.e., $\l$ satisfies equations~\eqref{eq:4}. Then Lemma~\ref{lem:trdeg_lambda} yields that 
\begin{equation}
\label{eq:7}
\l\in\overline{\Fp(\alpha_1,\beta)}. 
\end{equation}
Writing $f_\l^n(\alpha_2)=\beta$ as $P_{n,\alpha_2}(\l)=\beta$, i.e.,
\begin{equation}
\label{eq:5}
\l^{d^{n-1}}+c_{n,1}(\alpha_2)\l^{d^{n-1}-1}+\cdots + c_{n,d^{n-1}-1}(\alpha_2)\l+ \alpha_2^{d^n}=\beta,
\end{equation} 
where for each $i=1,\dots, d^{n-1}-1$, $c_{n,i}(\alpha_2)$ is a polynomial in $\alpha_2$ of degree at most $d\cdot i< d^n$ (according to Lemma~\ref{lem:poly}~(c)), we conclude that $\alpha_2\in \overline{\Fp(\l,\beta)}$. Then equation~\eqref{eq:7} yields the desired conclusion that $\alpha_2\in\overline{\Fp(\alpha_1,\beta)}$. By symmetry, we also have $\alpha_1\in\overline{\Fp(\alpha_2,\beta)}$. 
\end{proof}

\begin{lemma}
\label{lem:trdeg_1_2}
If $C(\alpha_1,\alpha_2;\beta)$ is infinite, then $\alpha_1\in\overline{\Fp(\alpha_2)}$ and $\alpha_2\in\overline{\Fp(\alpha_1)}$.
\end{lemma}

\begin{proof}
Suppose $C(\alpha_1,\alpha_2;\beta)$ is infinite. Due to the symmetry between $\alpha_1$ and $\alpha_2$, it suffices to prove that $\alpha_1\in \overline{\Fp(\alpha_2)}$ (since an identical argument would yield $\alpha_2\in\overline{\Fp(\alpha_1)}$). 

We let $L_1\colonequals\overline{\Fp(\alpha_2)}$ and $K\colonequals L_1(\alpha_1,\beta)$. If $K=L_1$, then $\alpha_1\in L_1=\overline{\Fp(\alpha_2)}$, as desired. Henceforth, we assume $K/L_1$ is a function field, either of transcendence degree equal to $1$, or of transcendence degree equal to $2$ (in which case, $\alpha_1$ and $\beta$ are algebraically independent over $L_1$). We argue by contradiction and assume that $\alpha_1\notin L_1$. 

Let $V$ be a smooth projective variety defined over $L_1$ whose function field equals $K$ (either $V$ is a curve if $\trdeg_{L_1}K=1$, or $V=\mathbb{P}^2$ if $\alpha_2$ and $\beta$ are algebraically independent over $L_1$). Let $\Omega_V$ be an inequivalent set of absolute values on $K$ corresponding to the irreducible divisors of $V$. For the function field $K/L_1$, we have:
\begin{equation}
\label{eq:9}
|\gamma|_w\le 1\text{ for each }w\in\Omega_V\text{ if and only if }\gamma\in L_1.
\end{equation}
Since $\alpha_1\notin L_1$, there exists $v\in\Omega_V$ such that $|\alpha_1|_v>1$. We also fix an extension of $|\cdot |_v$ to an absolute value on $\Kbar\subseteq \Lbar$.

Let $\l\in C(\alpha_1,\alpha_2;\beta)$; so, there exist $m,n\in\N$ such that $f_\l^m(\alpha_1)=f_\l^n(\alpha_2)=\beta$. Since we assumed that $C(\alpha_1,\alpha_2;\beta)$ is infinite, Lemma~\ref{lem:finitely_l} allows us to assume that both $m$ and $n$ are arbitrarily large; in particular, we may assume that 
\begin{equation}
\label{eq:8}
|\beta|_v<|\alpha_1|_v^{d^{{\rm min}\{m,n\}}}.
\end{equation}
The equation $f_\l^n(\alpha_2)=\beta$ yields $P_{n,\alpha_2}(\l)=\beta$ and so, using equation~\eqref{eq:3} along with the fact that $|\alpha_2|_v\le 1$ (see~\eqref{eq:9}), we conclude that 
\begin{equation}
\label{eq:10}
|\l|_v\le \max\{|\beta|_v,1\}^{\frac{1}{d^{n-1}}}.
\end{equation}
On the other hand, the equation $f_\l^m(\alpha_1)=\beta$ yields $P_{m,\alpha_1}(\l)=\beta$, i.e.,
\begin{equation}
\label{eq:11}
\l^{d^{m-1}}+\sum_{i=1}^{d^{m-1}} c_{m,i}(\alpha_1)\cdot \l^{d^{m-1}-i}=\beta.
\end{equation}
Lemma~\ref{lem:poly}~(c) yields that $\deg(c_{m,i})\le d\cdot i$ for each $i=1,\dots, d^{m-1}$; also, $c_{m,d^{m-1}}(\alpha_1)=\alpha_1^{d^m}$. Since  $|\beta|_v<|\alpha_1|_v^{d^m}$ by~\eqref{eq:8}, there exists some $i\in\{0,\dots, d^{m-1}-1\}$ such that
\begin{equation}
\label{eq:12}
\left|c_{m,i}(\alpha)\cdot \l^{d^{m-1}-i}\right|_v\ge |\alpha_1|_v^{d^m}.
\end{equation}
Because each $c_{m,i}$ is a polynomial with coefficients in $\Fp$ of degree at most $d\cdot i$ (according to Lemma~\ref{lem:poly}~(c)), we have that
\begin{equation}
\label{eq:13}
\left|c_{m,i}(\alpha_1)\right|_v\le |\alpha_1|_v^{d\cdot i}\text{ for each }i=0,\dots, d^{m-1}-1.
\end{equation}
Combining inequalities \eqref{eq:12} and \eqref{eq:13}, we obtain that 
\begin{equation}
\label{eq:14}
|\l|_v\ge |\alpha_1|_v^d.
\end{equation}
Next, combining inequalities \eqref{eq:10} and \eqref{eq:14}, we get
\begin{equation}
\label{eq:015}
\max\left\{|\beta|_v,1\right\}\ge |\l|_v^{d^{n-1}}\ge |\alpha_1|_v^{d^n}>1.
\end{equation}
Inequalities \eqref{eq:8} and \eqref{eq:015} provide a contradiction; so, we must have that $\alpha_1\in L_1$. 

This concludes our proof of Lemma~\ref{lem:trdeg_1_2}.
\end{proof}

\begin{proposition}
\label{prop:trdeg_right}
Let $\alpha_1,\alpha_2,\beta\in L$ and assume $C(\alpha_1,\alpha_2;\beta)$ is infinite. Then at least one of the following two conditions must hold:
\begin{itemize}
\item[(1)] $\alpha_1^d=\alpha_2^d$.
\item[(2)] $\beta\in\overline{\Fp(\alpha_1)}=\overline{\Fp(\alpha_2)}$.
\end{itemize}
\end{proposition}

\begin{proof}
First, Lemma~\ref{lem:trdeg_1_2} yields the equality $\overline{\Fp(\alpha_1)}=\overline{\Fp(\alpha_2)}$. So, letting $L_1\colonequals\overline{\Fp(\alpha_1)}$, it suffices to prove that if  $\beta\notin L_1$, then $\alpha_1^d=\alpha_2^d$.

We let $K_1=L_1(\beta)$; this is a rational function field of transcendence degree $1$ (since we assumed that $\beta\notin L_1$). We view $K_1$ as the function field of $\mathbb{P}^1$ over $L_1$. Let $|\cdot|_\infty$ be the absolute value on the function field $K_1/L_1$ corresponding to the place at infinity from $\mathbb{P}^1_{L_1}$; hence $|\beta|_\infty>1$. We fix an extension of $|\cdot |_\infty$ to an absolute value on the algebraic closure $\overline{K_1}$.

Let $j\in\{1,2\}$, let $\l\in\overline{K_1}$ and let $\ell\in\N$ such that $f_\l^\ell(\alpha_j)=\beta$. Then equation~\eqref{eq:3} yields
\begin{equation}
\label{eq:15}
\l^{d^{\ell-1}}+\sum_{i=1}^{d^{\ell-1}} c_{\ell,i}(\alpha_j)\cdot \l^{d^{\ell-1}-i}=\beta.
\end{equation}
Since $\left|c_{\ell,i}(\alpha_j)\right|_\infty\le 1$ (note that $\alpha_j\in L_1$), we get that 
\begin{equation}
\label{eq:16}
|\l|_\infty=|\beta|_\infty^{\frac{1}{d^{\ell-1}}}>1.
\end{equation} 
Now, let $\l\in C(\alpha_1,\alpha_2;\beta)$, and let $m,n\in\N $ such that $f_\l^m(\alpha_1)=f_\l^n(\alpha_2)=\beta$. Equation~\eqref{eq:16} yields that $$|\l|_\infty=|\beta|_\infty^{\frac{1}{d^{m-1}}}=|\beta|_\infty^{\frac{1}{d^{n-1}}};$$
since $|\beta|_\infty>1$, we have $m=n$. Consequently, $P_{m,\alpha_1}(\l)=P_{m,\alpha_2}(\l)=\beta$, i.e.
\begin{equation}
\label{eq:17}
\l^{d^{m-1}}+\sum_{i=1}^{d^{m-1}}c_{m,i}(\alpha_1)\cdot \l^{d^{m-1}-i}= \l^{d^{m-1}}+\sum_{i=1}^{d^{m-1}}c_{m,i}(\alpha_2)\cdot \l^{d^{m-1}-i},
\end{equation}
with the notation for $c_{m,i}$ as in Lemma~\ref{lem:poly}. Noting that $c_{m,d^{m}}(\alpha_j)=\alpha_j^{d^m}$ for $j=1,2$,
equation \eqref{eq:17} yields
\begin{equation}
\label{eq:18}
\sum_{i=1}^{d^{m-1}-1} \left(c_{m,i}(\alpha_1)-c_{m,i}(\alpha_2)\right)\cdot \l^{d^{m-1}-i} + \left(\alpha_1^{d^m}-\alpha_2^{d^m}\right)=0.
\end{equation}
Because $\l$ is transcendental over $L_1$ (due to \eqref{eq:16}), equation~\eqref{eq:18} yields that 
\begin{equation}
\label{eq:19}
\alpha_1^{d^m}=\alpha_2^{d^m}\text{ and}
\end{equation}
\begin{equation}
\label{eq:20}
c_{m,i}(\alpha_1)=c_{m,i}(\alpha_2)\text{ for each }i=1,\dots, d^{m-1}-1.
\end{equation}
Now, we write $d=p^r\cdot s$, for some nonnegative integer $r$ and some positive integer $s$, which is not divisible by $p$. If $s=1$, i.e., $d=p^r$, then equation~\eqref{eq:19} yields that $\alpha_1=\alpha_2$ and therefore, condition~(1) from the conclusion of Proposition~\ref{prop:trdeg_right} holds. Also, if $r=0$ (i.e., $d$ is \emph{not} divisible by $p$), we claim that $\alpha_1^d=\alpha_2^d$. Indeed, in this case, $c_{2,1}(\alpha_j)$ is given by the following formula for $j=1,2$:
\[c_{2,1}(\alpha_j)=\left\{\begin{array}{ccc}
d\alpha_j^d & \text{ if } & d>2,\\
2\alpha_j^2+1 & \text{ if } & d=2.
\end{array}\right.\]
An easy induction on $m$ shows that if $d>2$, then $c_{m,1}(\alpha_j)=d^{m-1}\alpha_j^d$ for $j=1,2$. Therefore, equation~\eqref{eq:20} yields $\alpha_1^d=\alpha_2^d$ (note that $p$ does not divide $d$ in this case because we assumed that $r=0$). A similar induction on $m$ shows that if $d=2$, then for $j=1,2$ and $m\ge 2$, we have that 
$$c_{m,1}(\alpha_j)=2^{m-2}\cdot \left(2\alpha_j^2+1\right)$$
and therefore, once again equation~\eqref{eq:20} yields $\alpha_1^2=\alpha_2^2$ (recall that in this case, $d=2$ and $p>2$).

So, from now on, we may assume $s\ge 2$ and $r\ge 1$. 
\begin{lemma}
\label{lem:exact_form}
With the above notation, we have that $c_{m,i}=0$ for $i=1,\dots, p^{(m-1)r}-1$, while 
$$c_{m,p^{(m-1)r}}(\alpha_j)=s^{(p^{r(m-1)}-1)/(p^r-1)}\alpha_j^{p^{rm}s}\text{ for $j=1,2$.}$$
\end{lemma}

\begin{proof}[Proof of Lemma~\ref{lem:exact_form}.]
The result follows readily by induction on $m$, using the recurrence relation~\eqref{eq:0}. The desired formula holds trivially also when $m=1$. To aid the reader's intuition for the general case, we show the computation for $m=2$. For the sake of simplifying our notation in Lemma~\ref{lem:exact_form},  we let $\alpha\colonequals \alpha_j$ (for $j=1,2$) and observe that
$$P_{2,\alpha}(\l) = \left(\alpha^d+\l\right)^d = \left(\alpha^{sp^r} + \l\right)^{p^r\cdot s}+\l= \left(\alpha^{sp^{2r}}+\l^{p^r}\right)^s+\l,$$
which yields that 
$$P_{2,\alpha}(\l)=\l^{p^r\cdot s}+s\cdot \alpha^{sp^{2r}}\cdot \l^{p^r(s-1)} + \text{ lower order terms,}$$
proving the claimed formula when $m=2$. For the inductive step, assume the formula holds for some $m\geq 1$:
$$P_{m,\alpha}(\l) = \l^{p^{(m-1)r}\cdot s^{m-1}} + s^{(p^{r(m-1)}-1)/(p^r-1)}\alpha^{p^{rm}s}\cdot \l^{p^{(m-1)r}\cdot \left(s^{m-1}-1\right)} + \text{ lower order terms.}$$
Then the recurrence relation \eqref{eq:0} yields
$$P_{m+1,\alpha}(\l) = \left(\l^{p^{(m-1)r}\cdot s^{m-1}} + s^{(p^{r(m-1)}-1)/(p^r-1)}\alpha^{p^{rm}s}\cdot \l^{p^{(m-1)r}\cdot \left(s^{m-1}-1\right)} + \cdots\right)^{p^r\cdot s}+\l.$$
Since the field has characteristic $p$, we can distribute the $p^r$ exponent, which gives
$$P_{m+1,\alpha}(\l)= \left(\l^{p^{mr}\cdot s^{m-1}} + s^{p^r\cdot (p^{r(m-1)}-1)/(p^r-1)}\alpha^{p^{r(m+1)}s}\cdot \l^{p^{mr}\cdot \left(s^{m-1}-1\right)} + \cdots\right)^s+\l.$$
We conclude that 
$$P_{m+1,\alpha}(\l) = \l^{p^{mr}\cdot s^m}+ s\cdot s^{p^r\cdot (p^{r(m-1)}-1)/(p^r-1)}\alpha^{p^{r(m+1)}s}\cdot \l^{p^{mr}\cdot s^{m-1}\cdot (s-1) + p^{mr}\cdot \left(s^{m-1}-1\right)} + \text{ $\cdots$, i.e.,}$$
$$P_{m+1,\alpha}(\l) = \l^{d^m} +s^{\frac{p^{rm}-1}{p^r-1}}\alpha^{p^{r(m+1)}s}\cdot \l^{d^m-p^{mr}}+\text{ lower order terms,} $$
which completes the inductive step and thus proves Lemma~\ref{lem:exact_form}.
\end{proof}
Applying Lemma~\ref{lem:exact_form}, the equation~\eqref{eq:20} implies that:
\begin{equation}
\label{eq:21}
s^{(p^{r(m-1)}-1)/(p^r-1)} \alpha_1^{p^{rm}s}=s^{(p^{r(m-1)}-1)/(p^r-1)}\alpha_2^{p^{rm}s}.
\end{equation}
Since $s\ne 0 $ in $\Fp$, equation~\eqref{eq:21} yields $\alpha_1^s=\alpha_2^s$ and therefore, $\alpha_1^d=\alpha_2^d$, as desired.

This concludes our proof of Proposition~\ref{prop:trdeg_right}.
\end{proof}

%%%%%%%%%%%%%%%%%%%%%%%%%%%%%%%%%%%%%%%%%%%%%%%%%%%%%%%%%%%%%%%%%%%%%%%%%%%%%%%
%%%%%%%%%%%%%%%%%%%%%%%%%%%%%%%%%%%%%%%%%%%%%%%%%%%%%%%%%%%%%%%%%%%%%%%%%%%%%%%

\section{Heights}
\label{sec:heights}

In this Section, we establish the framework of absolute values and heights needed for our proof of Theorem~\ref{thm:main}. Throughout Section~\ref{sec:heights}, we let $t$ be a transcendental element over $\Fpbar$. 

%%%%%%%%%%%%%%%%%%%%%%%%%%%%%%%%%%%%%%%%%%%%%%%%%%%%%%%%%%%%%%%%%%%%%%%%%%%%%%%

\subsection{Absolute values for the one-variable rational function field}

We let $\Omega_0\colonequals\Omega_{\Fpbar(t)}$ be the set of absolute values on $\Fpbar(t)$ defined as follows. For each  $c\in\Fpbar$, we have the unique absolute value $|\cdot |_{v_c}$ in $\Omega_0$ normalized  as follows: for any nonzero rational function $\frac{g_1}{g_2}$ (with $g_1,g_2\in\Fpbar[t]\setminus\{0\}$), we have $$\left|\frac{g_1}{g_2}\right|_{v_c}\colonequals e^{{\rm ord}_c(g_2)-{\rm ord}_c(g_1)},$$ 
where ${\rm ord}_c(g)$ is the order of vanishing at the point $c\in\Fpbar$ of any nonzero polynomial $g\in\Fpbar[t]$. Besides the above absolute values $|\cdot|_{v_c}$, we also have the absolute value $|\cdot|_{v_\infty}\in\Omega_0$ normalized as follows: for any nonzero rational function $\frac{g_1}{g_2}$ (with $g_1,g_2\in\Fpbar[t]\setminus\{0\}$), we have 
$$\left|\frac{g_1}{g_2}\right|_{v_\infty}=e^{\deg(g_1)-\deg(g_2)}.$$
It is immediate that for each nonzero rational function $g\in\Fpbar(t)$, we have the following product formula:
\begin{equation}
\label{eq:prod_formula}
\prod_{v\in \Omega_0}|g|_v=1.
\end{equation} 

%%%%%%%%%%%%%%%%%%%%%%%%%%%%%%%%%%%%%%%%%%%%%%%%%%%%%%%%%%%%%%%%%%%%%%%%%%%%%%%

\subsection{Extending the absolute values to the perfect closure}
\label{subsec:extending}

We let $L_0\subset\Lbar$ be the perfect closure of $\Fpbar(t)$, i.e., 
\begin{equation}
\label{eq:L}
L_0\colonequals\Fpbar\left(t, t^{1/p}, t^{1/p^2},\cdots, t^{1/p^n},\cdots\right).
\end{equation}
The field $L_0$ is \emph{perfect}, meaning all its finite extensions are separable.

Each absolute value $|\cdot |_v$ from $\Omega_0$ has a unique extension to an absolute value on $L_0$; we denote by $\Omega_{L_0}$ the set of all these extended absolute values. Once again, we have a product formula for each nonzero $\gamma\in L_0$:
\begin{equation}
\label{eq:prod_formula_2}
\prod_{v\in\Omega_{L_0}}|\gamma|_v=1.
\end{equation}

%%%%%%%%%%%%%%%%%%%%%%%%%%%%%%%%%%%%%%%%%%%%%%%%%%%%%%%%%%%%%%%%%%%%%%%%%%%%%%%

\subsection{Heights for any algebraic element}
\label{subsec:heights_0}

For any real number $u$, we define $\log^+|u|\colonequals\log\max\{|u|,1\}$.

We fix a finite extension $K$ of $L_0$. For any absolute value $|\cdot |_v\in \Omega_{L_0}$, there exist finitely many places $w$ of $K$ lying above the place $v$ of $L_0$, denoted by $w|v$. We normalize the corresponding absolute values $|\cdot |_w$ on $K$, and we denote by $\Omega\colonequals\Omega_K$ the corresponding set of absolute values. For each $\gamma\in L_0$ and each $v\in\Omega_{L_0}$, we have the following relation:
\begin{equation}
\label{eq:coherence}
|\gamma|_v=\prod_{\substack{w\in\Omega_{K}\\ w|v}}|\gamma|_w.
\end{equation}
Using equations~\eqref{eq:coherence}~and~\eqref{eq:prod_formula_2}, $K$ satisfies the product formula with respect to the absolute values from $\Omega=\Omega_K$, i.e.,
$$\prod_{w\in\Omega}|\gamma|_w=1\text{ for each nonzero }\gamma\in K.$$
We fix an algebraic closure $\Kbar$ of $K$. Then for each $w\in\Omega=\Omega_K$, we fix an extension of $|\cdot|_w$ to an absolute value on $\Kbar$.

For any $\gamma\in \Kbar$, we  define the \emph{Weil height} of $\gamma$ as follows:
\begin{equation}
\label{eq:height}
h(\gamma)\colonequals h_K(\gamma)\colonequals \frac{1}{[K(\gamma):K]}\cdot \sum_{\substack{\sigma:K(\gamma)\to \Kbar\\ \sigma|_{K}={\rm id}_{K}}}  \sum_{v\in\Omega}\log^+\left|\sigma(\gamma)\right|_v.
\end{equation}
\begin{remark}
\label{rem:K}
By construction, the Weil height depends on our choice of field $K$. In our proof of Theorem~\ref{thm:main}, we will choose a field $K$ so that it contains $\alpha_1,\alpha_2,\beta$ (for more details, see Section~\ref{sec:proof}).
\end{remark}

%%%%%%%%%%%%%%%%%%%%%%%%%%%%%%%%%%%%%%%%%%%%%%%%%%%%%%%%%%%%%%%%%%%%%%%%%%%%%%%%

\subsection{Canonical heights}
\label{subsec:heights}

For any $\l\in\Kbar$, we consider the polynomial $f_\l(z)=z^d+\l$ of degree $d\ge 2$. Let $\gamma\in \Kbar$. For each $v\in\Omega=\Omega_K$, we construct the \emph{local canonical height}: 
\begin{equation}
\label{eq:local_canonical}
\hhat_{\l,v}(\gamma)\colonequals\lim_{n\to\infty} \frac{\log^+\left|f_\l^n(\gamma)\right|_v}{d^n}.
\end{equation}
Equation~\eqref{eq:local_canonical} yields that for each $m\in\N$:
\begin{equation}
\label{eq:local_canonical_2}
\hhat_{\l,v}\left(f_\l^m(\gamma)\right)=d^m\cdot \hhat_{\l,v}(\gamma).
\end{equation}
We define the \emph{(global) canonical height} $\hhat_\l(\gamma)$ as follows:
\begin{equation}
\label{eq:global_canonical_2}
\hhat_\l(\alpha)=\lim_{n\to\infty} \frac{h\left(f_\l^n(\gamma)\right)}{d^n}.
\end{equation}
Equation~\eqref{eq:global_canonical_2} yields that for each $m\in\N$:
\begin{equation}
\label{eq:global_canonical_3}
\hhat_{\l}\left(f_\l^m(\gamma)\right)=d^m\cdot \hhat_{\l}(\gamma).
\end{equation}
For more details regarding the canonical height associated to a polynomial, see \cite[Section~3.4]{Silverman} and also \cite{C-S}. Furthermore, we have the following connection between the local and the global canonical heights.

\begin{lemma}
\label{lem:canonical_heights}
Let $\alpha\in K$ and $\l\in \Kbar$. Then the following holds:
\begin{equation}
\label{eq:50}
\hhat_\l(\alpha)=\frac{1}{[K(\lambda):K]}\cdot \sum_{\substack{\sigma:K(\lambda)\to \Kbar\\ \sigma|_K={\rm id}|_K}} \sum_{v\in\Omega_K} \hhat_{\sigma(\l),v}(\alpha).
\end{equation}
\end{lemma}

\begin{proof}
Employing equations~\eqref{eq:global_canonical_2}, \eqref{eq:local_canonical} and \eqref{eq:height}, we have:
\begin{equation}
\label{eq:521}
\hhat_\l(\alpha)=\frac{1}{[K(\l):K]}\cdot \lim_{n\to\infty} \frac{1}{d^n}\cdot \sum_{\substack{\sigma:K(\l)\to \Kbar\\ \sigma|_K={\rm id}|_K}} \sum_{v\in\Omega_K}\log^+\left|\sigma\left(f_\l^n(\alpha)\right)\right|_v
\end{equation}
and so, because $\alpha\in K$, we get
\begin{equation}
\label{eq:520}
\hhat_\l(\alpha)=\frac{1}{[K(\l):K]}\cdot \lim_{n\to\infty} \frac{1}{d^n}\cdot \sum_{\substack{\sigma:K(\l)\to \Kbar\\ \sigma|_K={\rm id}|_K}} \sum_{v\in\Omega_K}\log^+\left|f_{\sigma(\l)}^n(\alpha)\right|_v\text{; therefore,}
\end{equation}
\begin{equation}
\label{eq:52}
\hhat_\l(\alpha)=\frac{1}{[K(\l):K]}\cdot \sum_{v\in\Omega_K} \sum_{\substack{\sigma:K(\l)\to \Kbar\\ \sigma|_K={\rm id}|_K}} \hhat_{\sigma(\l),v}(\alpha).
\end{equation}
The limit and the corresponding summation over all $v\in\Omega_K$ in~\eqref{eq:520} can be interchanged, because the limit equals $0$ for all but finitely many places $v$ (see \cite[Lemma~3.4]{1}). 
\end{proof}

%%%%%%%%%%%%%%%%%%%%%%%%%%%%%%%%%%%%%%%%%%%%%%%%%%%%%%%%%%%%%%%%%%%%%%%%%%%%%%%
%%%%%%%%%%%%%%%%%%%%%%%%%%%%%%%%%%%%%%%%%%%%%%%%%%%%%%%%%%%%%%%%%%%%%%%%%%%%%%%

\section{Bounds for the canonical height for our family of polynomials}
\label{sec:bounds}

We continue with the framework for local and global canonical heights from Section~\ref{sec:heights}. Let 
$$L_0=\Fpbar\left(t,t^{\frac{1}{p}},t^{\frac{1}{p^2}},\cdots\right)$$
and let $K$ be a finite extension of $L_0$. We denote by $\Omega_0$ the set of absolute values on $L_0$ constructed as in Subsection~\ref{subsec:extending}. Also, we let $\Omega\colonequals\Omega_K$ be the normalized absolute values on $K$ corresponding to the places of $K$ lying above the places from $\Omega_0$ (see Subsection~\ref{subsec:heights_0}). In addition, we fix an extension of each $|\cdot|_v$ to an absolute value on $\Kbar$. We construct the local and global canonical heights (see equations~\eqref{eq:local_canonical}~and~\eqref{eq:global_canonical_2}) with respect to polynomials $f_\l(z)\colonequals z^d+\gamma$ (for $\gamma\in\Kbar$). The following result is proved in \cite[Lemma~3.4]{1}.

\begin{lemma}
\label{lem:easy}
Let $v\in \Omega$ and  let $\gamma,\l\in\Kbar$.  
\begin{enumerate}[label=(\roman*)]
\item\label{lem:easy:item:i} If $|\gamma|_v\le 1$ and $|\l|_v\le 1$, then $\hhat_{\l,v}(\gamma)=0$. 
\item\label{lem:easy:item:ii} If $|\gamma|_v^d>\max\{1,|\l|_v\}$, then $\hhat_{\l,v}(\gamma)=\log|\gamma|_v>0$.
\item\label{lem:easy:item:iii} If $|\l|_v>\max\{1,|\gamma|_v^{d}\}$, then $\hhat_{\l,v}(\gamma)=\frac{\log|\l|_v}{d}>0$. 
\end{enumerate}
\end{lemma}

Lemma~\ref{lem:easy} yields the following key inequality.
\begin{proposition}
\label{prop:useful_variation}
For $\alpha\in K$ and $\l\in\Kbar$, we have that 
\begin{equation}
\label{eq:variation}
\frac{h(\l)}{d}-h(\alpha) \le \hhat_\l(\alpha)\le \frac{h(\l)}{d}+h(\alpha).
\end{equation}
\end{proposition}

Proposition~\ref{prop:useful_variation} belongs to a long series of results regarding the variation of the canonical height in algebraic families. Most of these results are formulated over number fields (see \cite{C-S, D-M-2, G-M, Patrick, Silverman-variation-1, Silverman-variation-2, Silverman-variation-3, Tate}); by contrast, only a few such results are stated over function fields of positive characteristic (see \cite{G-JNT}). 

\begin{proof}[Proof of Proposition~\ref{prop:useful_variation}.]
We first prove the right-hand side inequality from~\eqref{eq:variation}. 

Let $v\in \Omega=\Omega_{K}$ and let $\sigma\colon K(\gamma)\to \Kbar$ be a field homomorphism fixing $K$ pointwise. We let 
\begin{equation}
\label{eq:M_v}
M_{v,\sigma}\colonequals\log^+|\alpha|_v + \frac{\log^+|\sigma(\l)|_v}{d}. 
\end{equation}
Lemma~\ref{lem:easy}~\ref{lem:easy:item:i} shows that 
\begin{equation}
\label{eq:23}
\text{if $|\alpha|_v,|\sigma(\l)|_v\le 1$, then $\hhat_{\sigma(\l),v}(\alpha)=0=M_{v,\sigma}$.} 
\end{equation}
Next, assume that $\max\{|\alpha|_v,|\sigma(\l)|_v\}>1$, i.e., $M_{v,\sigma}>0$.  
If $|\alpha|_v>|\sigma(\l)|_v^{\frac{1}{d}}$, then Lemma~\ref{lem:easy}~\ref{lem:easy:item:ii} yields that 
\begin{equation}
\label{eq:22}
\hhat_{\sigma(\l),v}(\alpha)=\log|\alpha|_v\le M_{v,\sigma}.
\end{equation} 
If $|\alpha|_v<|\sigma(\l)|^{\frac{1}{d}}$, then Lemma~\ref{lem:easy}~\ref{lem:easy:item:iii} guarantees that
\begin{equation}
\label{eq:24}
\hhat_{\sigma(\l),v}=\frac{\log|\sigma(\l)|_v}{d}\le M_{v,\sigma}.
\end{equation}
Now, if $|\alpha|_v=|\sigma(\l)|_v^{\frac{1}{d}}>1$, then we get $\left|f_{\sigma(\l)}(\alpha)\right|_v\le |\alpha|_v^d=|\sigma(\l)|_v$ and so,
\begin{equation}
\label{eq:25}
\log\left|f_{\sigma(\l)}(\alpha)\right|_v< d\cdot M_{v,\sigma}.
\end{equation}
For each $n\ge 1$, we have 
\begin{equation}
\label{eq:26}
\left|f_{\sigma(\l)}^{n+1}(\alpha)\right|_v\le \max\left\{\left|f^n_{\sigma(\l)}(\alpha)\right|_v^d,|\sigma(\l)|_v\right\}.
\end{equation}
Employing inequalities~\eqref{eq:25}~and~\eqref{eq:26}, a simple induction on $n$  results in:
\begin{equation}
\label{eq:27}
\left|f_{\sigma(\l)}^n(\alpha)\right|_v\le d^n\cdot M_{v,\sigma}.
\end{equation}
Inequality \eqref{eq:27} and the definition~\eqref{eq:local_canonical} of the local canonical height together give:
\begin{equation}
\label{eq:28}
\hhat_{\sigma(\l),v}(\alpha)\le M_{v,\sigma}.
\end{equation}
Using equations~\eqref{eq:28}, \eqref{eq:24}, \eqref{eq:23} and \eqref{eq:22}, along with Lemma~\ref{lem:canonical_heights}, we get that
$$
\hhat_{\l}(\alpha)= \frac{1}{[K(\l):K]}\cdot \sum_{\substack{\sigma:K(\l)\to\Kbar\\ \sigma|_K={\rm id}|_K}} 
 \sum_{v\in\Omega} \hhat_{v,\sigma(\l)}(\alpha)\le \frac{1}{[K(\l):K]}\cdot \sum_{\substack{\sigma:K(\l)\to\Kbar\\ \sigma|_K={\rm id}|_K}} 
 \sum_{v\in\Omega} M_{v,\sigma} \text{ and so,}
$$

$$
\hhat_\l(\alpha)\le \frac{1}{[K(\l):K]}\cdot \sum_{\substack{\sigma:K(\l)\to\Kbar\\ \sigma|_K={\rm id}|_K}} \sum_{v\in\Omega} \left(\log^+|\alpha|_v +\frac{\log^+|\sigma(\l)|_v}{d}\right) \text{, which yields}$$

$$\hhat_\l(\alpha)\le  \left(\sum_{v\in\Omega} \log^+|\alpha|_v\right) + \frac{1}{d\cdot [K(\l):K]}\cdot \sum_{\substack{\sigma:K(\l)\to\Kbar\\ \sigma|_K={\rm id}|_K}} \sum_{v\in\Omega} \log^+|\sigma(\l)|_v \text{; hence}
$$

\begin{equation}
\label{eq:29}
\hhat_\l(\alpha)\le h(\alpha)+\frac{h(\l)}{d}\text{, thus proving the right-hand side of \eqref{eq:variation}.}
\end{equation}
Next, we will establish the inequality from the left-hand side of~\eqref{eq:variation}. Again, we obtain suitable inequalities for each place $v\in\Omega$; this time, we define $N_{v,\sigma}\colonequals\frac{\log^+|\sigma(\l)|_v}{d}-\log^+|\alpha|_v$. Immediately, we note that
\begin{equation}
\label{eq:30}
\text{if $\log^+|\sigma(\l)|_v\le d\cdot \log^+|\alpha|_v$, then $\hhat_{\sigma(\l),v}(\alpha)\ge 0\ge N_v$.} 
\end{equation}
On the other hand, if $|\sigma(\l)|_v>\max\left\{1,|\alpha|^d_v\right\}$, then Lemma~\ref{lem:easy}~\ref{lem:easy:item:iii} yields that 
\begin{equation}
\label{eq:31}
\hhat_{\sigma(\l),v}(\alpha)= \frac{\log|\sigma(\l)|_v}{d}=\frac{\log^+|\sigma(\l)|_v}{d}\ge N_v.
\end{equation}
Combining inequalities~\eqref{eq:30}~and~\eqref{eq:31} along with Lemma~\ref{lem:canonical_heights}, we obtain
$$
\hhat_\l(\alpha) = \frac{1}{[K(\l):K]}\cdot \sum_{\substack{\sigma:K(\l)\to\Kbar\\ \sigma|_K={\rm id}|_K}} \sum_{v\in\Omega} \hhat_{v,\sigma(\l)}(\alpha) \ge \frac{1}{[K(\l):K]}\cdot \sum_{\substack{\sigma:K(\l)\to\Kbar\\ \sigma|_K = {\rm id}|_K}}\sum_{v\in\Omega} N_{v,\sigma}$$
and so,
$$
\hhat_\l(\alpha)\ge \frac{1}{[K(\l):K]}\cdot \sum_{\substack{\sigma:K(\l)\to \Kbar\\ \sigma|_K = {\rm id}|_K}} \sum_{v\in\Omega} \left(\frac{\log^+|\sigma(\l)|_v}{d}-\log^+|\alpha|_v\right)\text{, which yields}$$
$$\hhat_\l(\alpha)\ge \left(\frac{1}{d\cdot [K(\l):K]}\cdot \sum_{\substack{\sigma:K(\l)\to\Kbar\\ \sigma|_K = {\rm id}|_K}} \sum_{v\in\Omega} \log^+|\sigma(\l)|_v \right)
-\left(\sum_{v\in\Omega}\log^+|\alpha|_v\right)\text{; hence}
$$
\begin{equation}
\label{eq:32}
\hhat_\l(\alpha)\ge  \frac{h(\l)}{d}-h(\alpha)\text{, as desired for the left-hand side of \eqref{eq:variation}.}
\end{equation}
Inequalities \eqref{eq:29} and \eqref{eq:32} establish the desired conclusion of Proposition~\ref{prop:useful_variation}.
\end{proof}

The following result is instrumental in our proof for Theorem~\ref{thm:main}.
\begin{proposition}
\label{prop:useful_2}
Let $d\ge 2$ be an integer,  let $\alpha,\beta\in K$ and let $\l\in\Kbar$. Let $f_\l(z)\colonequals z^d+\l$ and let $m\in\N$. If $f_\l^m(\alpha)=\beta$, then 
\begin{equation}
\label{eq:33}
\hhat_\l(\alpha)\le \frac{2h(\alpha)+2h(\beta)}{d^m}.
\end{equation}
\end{proposition}

\begin{remark}
\label{2-rem}
It is worth noting that variants of Proposition~\ref{prop:useful_2} were obtained in other settings, such as for families of elliptic curves over number fields (see \cite[Proposition~3.1]{GHT-PJM}) and also for families of Drinfeld modules (see \cite[Proposition~4.4]{G-JNT}). In each case, one combines a result on the variation of the canonical height in the respective family of maps (in our case, Proposition~\ref{prop:useful_variation}) with the precise growth rate of the canonical height in a given orbit (see equation~\eqref{eq:global_canonical_3}).
\end{remark}

\begin{proof}[Proof of Proposition~\ref{prop:useful_2}.]
From Proposition~\ref{prop:useful_variation}, we have the inequalities:
\begin{equation}
\label{eq:34}
\hhat_\l(\alpha)\ge \frac{h(\l)}{d}-h(\alpha)\text{ and }\hhat_\l(\beta)\le \frac{h(\l)}{d}+h(\beta).
\end{equation}
Using the condition $f_\l^m(\alpha)=\beta$ and the identity $d^m\hhat_\l(\alpha)=\hhat_\l(\beta)$ from equation~\eqref{eq:global_canonical_3}, we combine the inequalities in~\eqref{eq:34} to obtain:
\begin{equation}
\label{eq:35}
d^m\cdot \left(\frac{h(\l)}{d}-h(\alpha)\right)\le \frac{h(\l)}{d}+h(\beta).
\end{equation}
Rearranging the terms in \eqref{eq:35} gives an upper bound on the height of $\lambda$:
\begin{equation}
\label{eq:36}
h(\l)\le \frac{d^m\cdot h(\alpha)+h(\beta)}{\frac{d^m-1}{d}}.
\end{equation}
Using the fact that $d^m-1\ge \frac{d^m}{2}$ (since $d\ge 2$ and $m\ge 1$), inequality~\eqref{eq:36} becomes
\begin{equation}
\label{eq:37}
h(\l)\le \frac{2d^{m+1}\cdot h(\alpha)+2d\cdot h(\beta)}{d^m}.
\end{equation}
Combining inequality~\eqref{eq:37} with the second inequality in~\eqref{eq:34} leads to
\begin{equation}
\label{eq:38}
\hhat_\l(\beta)\le \frac{h(\l)}{d}+h(\beta)\le \frac{2d^{m+1}\cdot h(\alpha)+2d\cdot h(\beta)}{d^{m+1}} + h(\beta).
\end{equation}
Since $d^{m+1}\ge 2d$ for $d\ge 2$, inequality~\eqref{eq:38} implies
\begin{equation}
\label{eq:39}
\hhat_\l(\beta)\le 2h(\alpha)+ h(\beta) + h(\beta)=2h(\alpha)+2h(\beta).
\end{equation}
Finally, using $\hhat_\l(\alpha)=\frac{\hhat_\l(\beta)}{d^m}$ from equation~\eqref{eq:global_canonical_3} along with inequality~\eqref{eq:39}, we obtain
$$\hhat_\l(\alpha)\le \frac{2h(\alpha)+2h(\beta)}{d^m},$$
as desired.
\end{proof}

%%%%%%%%%%%%%%%%%%%%%%%%%%%%%%%%%%%%%%%%%%%%%%%%%%%%%%%%%%%%%%%%%%%%%%%%%%%%%%%
%%%%%%%%%%%%%%%%%%%%%%%%%%%%%%%%%%%%%%%%%%%%%%%%%%%%%%%%%%%%%%%%%%%%%%%%%%%%%%%

\section{Proof of Theorem~\ref{thm:main}}
\label{sec:proof}

We work with the notation and hypotheses from Theorem~\ref{thm:main}. In particular, we have the family of polynomials $f_\l(z)=z^d+\l$ (with a fixed $d\ge 2$) parameterized by $\l\in\Lbar$, where $L$ is a field of characteristic $p$. Furthermore, for some $\alpha_1,\alpha_2,\beta\in L$, the set
\begin{equation}
\label{eq:01}
C(\alpha_1,\alpha_2;\beta)=\left\{\l\in\Lbar\colon\text{ there exist $m,n\in\N$ such that $f_\l^m(\alpha_1)=f_\l^n(\alpha_2)=\beta$}\right\}
\end{equation} 
is assumed to be infinite. Our goal is to prove that one of the alternatives~(i)-(iii) in Theorem~\ref{thm:main} must hold. 
We split our analysis into several subsections.

%%%%%%%%%%%%%%%%%%%%%%%%%%%%%%%%%%%%%%%%%%%%%%%%%%%%%%%%%%%%%%%%%%%%%%%%%%%%%%%

\subsection{Reduction to a field of transcendence degree one}

We first prove that we may assume $\alpha_1$ is transcendental over $\Fp$.

\begin{proposition}
\label{prop:trdeg_1}
Under the hypotheses from Theorem~\ref{thm:main}, at least one of the following three alternatives must hold:
\begin{itemize}
\item[(A)] $\alpha_1^d=\alpha_2^d$.
\item[(B)] $\alpha_1$ is transcendental over $\Fp$, while $\alpha_2,\beta\in \overline{\Fp(\alpha_1)}\subseteq \Lbar$.
\item[(C)] $\alpha_1,\alpha_2,\beta\in\Fpbar$.
\end{itemize}
\end{proposition}

\begin{proof}
Either $\alpha_1$ belongs to $\Fpbar$ or $\alpha_1$ is is transcendental over $\Fp$. First, if $\alpha_1\in\Fpbar\subseteq \Lbar$, then $\alpha_2\in\Fpbar$ by Lemma~\ref{lem:trdeg_1_2}. Applying Proposition~\ref{prop:trdeg_right} yields  $\alpha_1^d=\alpha_2^d$ which is alternative~(A), or $\beta\in \Fpbar$ which is alternative~(C).

Next, assume $\alpha_1$ is transcendental over $\Fp$. Lemma~\ref{lem:trdeg_1_2} yields $\alpha_2\in\overline{\Fp(\alpha_1)}$. Furthermore, Proposition~\ref{prop:trdeg_right} yields $\beta\in\overline{\Fp(\alpha_1)}$ which is alternative~(B), or $\alpha_1^d=\alpha_2^d$ which is alternative~(A). The proof of Proposition~\ref{prop:trdeg_1} is complete in all cases.  \end{proof}

The alternatives~(A)~and~(C) from Proposition~\ref{prop:trdeg_1} match the alternatives~(i)~and~(iii) from the conclusion of Theorem~\ref{thm:main}.  In light of Lemma~\ref{lem:trdeg_1_2} and Proposition~\ref{prop:trdeg_right}, we may henceforth assume that:
\begin{equation}
\label{eq:assumption}
\alpha_1\notin\Fpbar\text{ and also, $\alpha_2,\beta\in\overline{\Fp(\alpha_1)}$.} 
\end{equation}
We will prove that under hypothesis~\eqref{eq:assumption}, at least one of the alternatives~(i)~or~(ii) from Theorem~\ref{thm:main} must hold, i.e.,
\begin{itemize}
\item[(i)]  $\alpha_1^d=\alpha_2^d$; or 
\item[(ii)] $d$ is a power of the characteristic $p$ of our field $L$, and $\alpha_1-\beta,\alpha_2-\beta\in\Fpbar$.
\end{itemize}

%%%%%%%%%%%%%%%%%%%%%%%%%%%%%%%%%%%%%%%%%%%%%%%%%%%%%%%%%%%%%%%%%%%%%%%%%%%%%%%

\subsection{Two sequences of heights tending to zero}

Working under the assumption~\eqref{eq:assumption}, let $t\colonequals\alpha_1$ (which is transcendental over $\Fpbar$) and also, let  
$$L_0\colonequals\Fpbar\left(t,t^{\frac{1}{p}},t^{\frac{1}{p^2}},\cdots, t^{\frac{1}{p^k}},\cdots \right).$$
Then, according to \eqref{eq:assumption}, there exists a finite extension $K$ of $L_0$ such that $\alpha_1,\alpha_2,\beta\in K$. As a consequence of Lemma~\ref{lem:trdeg_lambda}, we have that $C(\alpha_1,\alpha_2;\beta)\subset\Kbar$. Next, we obtain an easy consequence of Proposition~\ref{prop:useful_2}, which is key for our argument.

\begin{proposition}
\label{prop:useful}
There exist an infinite subset $\{\l_k\}$ in $\Kbar$ such that 
\begin{equation}
\label{eq:tending_to_0}
\lim_{k\to\infty}\hhat_{\l_k}(\alpha_1)=\lim_{k\to\infty}\hhat_{\l_k}(\alpha_2)=0.
\end{equation}
\end{proposition}

\begin{proof}
Using Lemma~\ref{lem:finitely_l}, for each $\ell\in\N$ there exist finitely many $\l\in\Kbar$ such that $f_\l^\ell(\alpha_1)=\beta$ (or $f_\l^\ell(\alpha_2)=\beta$). Therefore, there exists an infinite sequence $\{\l_k\}_{k\in\N}$ of elements in $C(\alpha_1,\alpha_2;\beta)\subset \Kbar$ for which the corresponding exponents $m_k,n_k\in\N$ satisfying
\begin{equation}
\label{eq:41}
f_{\l_k}^{m_k}(\alpha_1)=f_{\l_k}^{n_k}(\alpha_2)=\beta,
\end{equation}
must also satisfy
\begin{equation}
\label{eq:42}
\lim_{k\to\infty} m_k=\lim_{k\to\infty}n_k=\infty.
\end{equation}
Now, applying Proposition~\ref{prop:useful_2} to the two relations in ~\eqref{eq:41} yields the following inequalities:
\begin{equation}
\label{eq:43}
\hhat_{\l_k}(\alpha_1)\le \frac{2h(\alpha_1)+2h(\beta)}{d^{m_k}}\text{ and }\hhat_{\l_k}(\alpha_2)\le \frac{2h(\alpha_2)+2h(\beta)}{d^{n_k}}.
\end{equation}
Inequalities \eqref{eq:43} combined with equation~\eqref{eq:42} lead to the desired conclusion.
\end{proof}

%%%%%%%%%%%%%%%%%%%%%%%%%%%%%%%%%%%%%%%%%%%%%%%%%%%%%%%%%%%%%%%%%%%%%%%%%%%%%%

\subsection{Equality for all local heights}
First, for each $v\in\Omega_K$, we fix an extension of $|\cdot|_v$ to an absolute value on the entire $\Kbar$; also, we define the local canonical heights $\hhat_{\l,v}$ as in \eqref{eq:local_canonical}. Next, we state \cite[Theorem~4.1]{1}, which is instrumental in our proof.

\begin{theorem}
\label{thm:4.1}
With the above notation for $K,\alpha_1,\alpha_2$, assume there exist infinitely many $\{\l_k\}_{k\in\N}$ such that $\lim_{k\to\infty}\hhat_{\l_k}(\alpha_1)=\lim_{k\to\infty}\hhat_{\l_k}(\alpha_2) = 0$. Then for each  $\l\in \Kbar$ and for each $v\in\Omega_{K}$, we have that 
\begin{equation}
\label{eq:44}
\hhat_{\l,v}(\alpha_1)=\hhat_{\l,v}(\alpha_2).
\end{equation}
\end{theorem}

Proposition~\ref{prop:useful} yields the existence of an infinite sequence $\{\l_k\}_{k\in\N}$ such that the hypotheses in Theorem~\ref{thm:4.1} are met. Therefore, we conclude that equation~\eqref{eq:44} holds for each absolute value $|\cdot|_v$ and for each $\l\in\Kbar$. 

%%%%%%%%%%%%%%%%%%%%%%%%%%%%%%%%%%%%%%%%%%%%%%%%%%%%%%%%%%%%%%%%%%%%%%%%%%%%%%%

\subsection{Finishing the proof of Theorem~\ref{thm:main} in the special case when the degree is a power of the characteristic}

In this Subsection, we work under the additional hypothesis that $d=p^\ell$ for some $\ell\in\N$; also, we know that equation~\eqref{eq:44} holds. With this new assumption, we will prove that 
\begin{itemize}
\item either $\alpha_1=\alpha_2$, i.e., conclusion~(i) from Theorem~\ref{thm:main} holds.
\item or $\alpha_1-\beta,\alpha_2-\beta\in\Fpbar$, i.e., conclusion~(ii) from Theorem~\ref{thm:main} holds. 
\end{itemize}
Since $d=p^\ell$, the iterates of $f_\l(z)=z^d+\ell$ have the following explicit form:
\begin{equation}
\label{eq:121}
f_\l^n(z)=z^{p^{\ell n}}+ \sum_{i=0}^{n-1}\l^{p^{i\ell}},
\end{equation}
for each $n\in\N$. The formula~\eqref{eq:121} will help in the proof of the following result.
\begin{proposition}
\label{prop:diff_fpbar}
With the above assumptions, we have $\alpha_1-\alpha_2\in\Fpbar$.
\end{proposition}

\begin{proof}
We argue by contradiction. Assume $\alpha_1-\alpha_2\notin\Fpbar$, which means there exists some $v\in\Omega_K$ such that $|\alpha_1-\alpha_2|_v>1$. Since $|\alpha_1-\alpha_2|_v\le \max\{|\alpha_1|_v,|\alpha_2|_v\}$, we may assume without loss of generality that $|\alpha_2|_v>1$.  We choose $n\in\N$ large enough so that
\begin{equation}
\label{eq:60}
|\alpha_1-\alpha_2|_v^{d^n}>|\alpha_2|_v.
\end{equation}
We let $\lambda\in\Kbar$ such that $f_\l^n(\alpha_2)=0$. By equation~\eqref{eq:121}, we have: 
\begin{equation}
\label{eq:61}
\alpha_2^{p^{n\ell}}+\l^{p^{\ell(n-1)}}+ \l^{p^{\ell(n-2)}} + \cdots + \l^{p^\ell}+\l=0.
\end{equation} 
First, we compute the local canonical height of $\alpha_2$. Since $|\alpha_2|_v>1$, equation~\eqref{eq:61} implies:
\begin{equation}
\label{eq:62}
|\l|_v=|\alpha_2|_v^{\frac{1}{p^\ell}}=|\alpha_2|_v^{\frac{1}{d}}.
\end{equation}
As $|f_\l^n(\alpha_2)|_v=0<|\l|_v^{1/d}$, we apply Lemma~\ref{lem:easy}~\ref{lem:easy:item:iii} and use \eqref{eq:62} to find:
\begin{equation}
\label{eq:63}
\hhat_{\l,v}\left(f_\l^n(\alpha_2)\right)=\frac{\log|\l|_v}{d}= \frac{\log|\alpha_2|_v}{d^2}.
\end{equation}
Combining equations \eqref{eq:63} and \eqref{eq:local_canonical_2}, we obtain:
\begin{equation}
\label{eq:64}
\hhat_{\l,v}(\alpha_2)=\frac{\log|\alpha_2|_v}{d^{n+2}}.
\end{equation}
Next, we compute the local canonical height of $\alpha_1$. Using equations~\eqref{eq:121} and \eqref{eq:61}, we express $f_\l^n(\alpha_1)$ as follows:
\begin{equation}
\label{eq:65}
f_\l^n(\alpha_1)=\alpha_1^{p^{n\ell}}+\sum_{i=0}^{n-1}\l^{p^{i\ell}}= \alpha_1^{p^{n\ell}} -\alpha_2^{p^{n\ell}}= \left(\alpha_1-\alpha_2\right)^{d^n}.
\end{equation} 
Using equations \eqref{eq:65}, \eqref{eq:60} and \eqref{eq:62}, we deduce that 
\begin{equation}
\label{eq:66}
\left|f_\l^n(\alpha_1)\right|_v=|\alpha_1-\alpha_2|_v^{d^n}>|\alpha_2|_v> |\alpha_2|_v^{\frac{1}{d}}=|\l|_v.
\end{equation}
Equation \eqref{eq:66} allows us to apply Lemma~\ref{lem:easy}~\ref{lem:easy:item:ii}, which yields
\begin{equation}
\label{eq:67}
\hhat_{\l,v}\left(f_\l^n(\alpha_1)\right)=\log\left|f_\l^n(\alpha_1)\right|_v= d^n\log|\alpha_1-\alpha_2|_v.
\end{equation}
From equations \eqref{eq:67} and \eqref{eq:local_canonical_2}, it follows that
\begin{equation}
\label{eq:68}
\hhat_{\l,v}(\alpha_1)=\log|\alpha_1-\alpha_2|_v.
\end{equation} 
Finally, comparing the local heights from equations~\eqref{eq:64}~and~\eqref{eq:68}, our initial choice of $n$ in inequality~\eqref{eq:60} shows that $\hhat_{\l,v}(\alpha_1)>\hhat_{\l,v}(\alpha_2)$. This contradicts equation~\eqref{eq:44}. Therefore, there is no $v\in\Omega_K$ such that $|\alpha_1-\alpha_2|_v>1$. Hence, $\alpha_1-\alpha_2\in\Fpbar$, as desired.
\end{proof}

We recall the existence of the infinite sequence $\{\l_k\}_{k\in\N}$ in $\Kbar$ for which there exist corresponding exponents $m_k,n_k\in\N$ such that 
\begin{equation}
\label{eq:75}
f_{\l_k}^{m_k}(\alpha_1)=f_{\l_k}^{n_k}(\alpha_2)=\beta
\end{equation}
for each $k\in \N$. The next lemma shows that if $m_k=n_k$ for some $k\in\N$, then conclusion~(i) from Theorem~\ref{thm:main} must hold.  
\begin{lemma}
\label{lem:easy_perm}
If $m_k=n_k$ for some $k\in\N$, then $\alpha_1=\alpha_2$.
\end{lemma}

\begin{proof}
Since $d=p^\ell$, we have that $f_\l(z)$ is a permutation polynomial on $\Kbar$ (for each $\l\in\Kbar$); it follows that $f_\l^m$ also induces a permutation on $\Kbar$ for each $m\in\N$. Therefore, if $m_k=n_k$, the condition $f_{\l_k}^{m_k}(\alpha_1)=f_{\l_k}^{n_k}(\alpha_2)$ becomes $f_{\l_k}^{m_k}(\alpha_1)=f_{\l_k}^{m_k}(\alpha_2)$, which implies $\alpha_1=\alpha_2$.  
\end{proof}

Lemma~\ref{lem:easy_perm} thus shows that alternative~(i) from the conclusion of Theorem~\ref{thm:main} holds if $m_k=n_k$ for some $k\in\N$. From now on, we assume that $m_k\ne n_k$ for all $k\in\N$.

The following result is an easy application of the formula for $f_\l^n(z)$ from equation~\eqref{eq:121}. Furthermore, Lemma~\ref{lem:prep_0} will also be used in Section~\ref{sec:ii} in the proof of Theorem~\ref{thm:converse_ii}. 
\begin{lemma}
\label{lem:prep_0}
Let $\alpha,\l\in\Kbar$ and let $\gamma\in\Fpbar$. If for some $m\in\N$ the following holds:
\begin{equation}
\label{eq:72}
f_\l^m(\alpha)=\alpha+\gamma,
\end{equation}
then $\alpha$ and $\alpha+\gamma$ are periodic under the action of $f_\l$.
\end{lemma}

\begin{proof}
Equation~\eqref{eq:121} and our hypothesis~\eqref{eq:72} together yield:
\begin{equation}
\label{eq:71}
f_\l^{2m}(\alpha)= f_\l^m(\alpha+\gamma)=\alpha^{p^{m\ell}}+\gamma^{p^{m\ell}}+\sum_{i=0}^{m-1} \l^{p^{i\ell}}=\gamma^{p^{m\ell}}+f_\l^m(\alpha)=\gamma^{p^{m\ell}}+\gamma+\alpha.
\end{equation}
An easy induction (iterating the computation from \eqref{eq:71} and employing equation~\eqref{eq:72}) shows that for each $n\ge 1$,
\begin{equation}
\label{eq:73}
f_\l^{mn}(\alpha+\gamma)=\sum_{j=0}^{n-1}\gamma^{p^{jm\ell}}+\alpha.
\end{equation}
We choose $r\in\N$ such that $\gamma\in \mathbb{F}_{p^r}$. Equation~\eqref{eq:73} then implies that for each $n\ge 1$,
\begin{equation}
\label{eq:74}
f_\l^{mn}(\alpha+\gamma)-\alpha\in \mathbb{F}_{p^r}.
\end{equation}
Equation~\eqref{eq:74} shows that the sequence $\{f_\l^{mn}(\alpha+\gamma)-\alpha\}_{n\geq 1}$ takes only finitely many values. Therefore, there exist positive integers $\ell_1<\ell_2$ such that $f_\l^{m\ell_1}(\alpha+\gamma)=f_\l^{m\ell_2}(\alpha+\gamma)$. Hence, $\alpha+\gamma$ is preperiodic under the action of $f_\l$. By equation~\eqref{eq:72}, $\alpha$ must also be preperiodic under the action of $f_\l$. Finally, because $f_\l$ induces a permutation on $\Lbar$ (as explained in the proof of Lemma~\ref{lem:easy_perm}), any preperiodic point is necessarily periodic. Thus, both $\alpha$ and $\alpha+\gamma$ are periodic points for $f_\l$, as desired.
\end{proof}

The following result is an immediate consequence of Lemma~\ref{lem:prep_0}.
\begin{lemma}
\label{lem:prep}
With the notation as in equation~\eqref{eq:75}, for each $k\in\N$, the points $\alpha_1,\alpha_2$, and $\beta$ are periodic under the action of $f_{\l_k}$.
\end{lemma}

\begin{proof}
Let $k\in\N$. By our assumption that $m_k\ne n_k$, we may assume without loss of generality that $m_k<n_k$. Because $f_\l(z)$ induces a permutation on $\Kbar$, the relation~\eqref{eq:75} implies
\begin{equation}
\label{eq:76}
\alpha_1=f_{\l_k}^{n_k-m_k}(\alpha_2).
\end{equation}
Since $n_k-m_k\ge 1$ and $\alpha_1-\alpha_2\in\Fpbar$ by Proposition~\ref{prop:diff_fpbar},  equation~\eqref{eq:76} provides the hypothesis needed to apply Lemma~\ref{lem:prep_0}. The lemma then implies that $\alpha_1$ and $\alpha_2$ are periodic under the action of $f_{\l_k}$. From~\eqref{eq:75}, $\beta$ is also periodic under the action of $f_{\l_k}$. \end{proof}

We already proved that $\alpha_1-\alpha_2\in\Fpbar$; the following result shows that $\alpha_1-\beta\in\Fpbar$ as well, which gives the complete picture for alternative~(ii) from the conclusion of Theorem~\ref{thm:main}.
\begin{proposition}
\label{prop:diff_fpbar_2}
We must have $\alpha_1-\beta\in\Fpbar$.
\end{proposition}

\begin{proof}
Fix $k\in\mathbb{N}$, and let $\l\colonequals\l_k$. According to Lemma~\ref{lem:prep}, $\alpha_1$ and $\beta$ are periodic under the action of $f_\l$. Using formula~\eqref{eq:121}, we compute:
\begin{equation}
\label{eq:78}
f_\l^n(\alpha_1)-f_\l^n(\beta)=\left(\alpha_1-\beta\right)^{p^{n\ell}}\text{ for each }n\in\N.
\end{equation}
Since both sequences $\{f_\l^n(\alpha_1)\}_{n\geq 1}$ and $\{f_\l^n(\beta)\}_{n\geq 1}$ have finitely many elements, the left-hand side of~\eqref{eq:78} takes only finitely many values as $n\in\N$ varies. Thus, the right-hand side of~\eqref{eq:78} also takes finitely many values,  that is, $\left(\alpha_1-\beta\right)^{p^{n_1\ell}}=\left(\alpha_1-\beta\right)^{p^{n_2\ell}}$ for some positive integers $n_1<n_2$. Consequently, we obtain $\alpha_1-\beta\in\Fpbar$.
\end{proof}

To summarize, Propositions~\ref{prop:diff_fpbar}~and~\ref{prop:diff_fpbar_2} yield that when $d=p^\ell$,  the hypotheses in Theorem~\ref{thm:main} imply that at least one of the alternatives~(i)-(iii) from its conclusion must hold. Therefore, for the remainder of the proof, we assume $d$ is not a power of $p$.

%%%%%%%%%%%%%%%%%%%%%%%%%%%%%%%%%%%%%%%%%%%%%%%%%%%%%%%%%%%%%%%%%%%%%%%%%%%%%%%

\subsection{Conclusion for our proof of Theorem~\ref{thm:main}}

We have now reduced the proof to the case where $\alpha_1\notin\Fpbar$ and $d$ is not a power of the characteristic $p$. The following result, which is \cite[Proposition~5.1]{1}, provides the final step.
\begin{theorem}
\label{thm:5.1}
With the above notation for $K,\alpha_1,\alpha_2$, assume equality~\eqref{eq:44} holds for each absolute value $|\cdot|_v$ and for each $\l\in\Kbar$. If $d$ is not a power of $p=\operatorname{char}(K)$, and at least one of $\alpha_1, \alpha_2$ is not in $\Fpbar$, then $\alpha_1^d=\alpha_2^d$.
\end{theorem}

This completes the proof of Theorem~\ref{thm:main}.

%%%%%%%%%%%%%%%%%%%%%%%%%%%%%%%%%%%%%%%%%%%%%%%%%%%%%%%%%%%%%%%%%%%%%%%%%%%%%%%
%%%%%%%%%%%%%%%%%%%%%%%%%%%%%%%%%%%%%%%%%%%%%%%%%%%%%%%%%%%%%%%%%%%%%%%%%%%%%%%

\section{Proof of Theorem~\ref{thm:converse_i}}
\label{sec:i}

In this Section, we prove Theorem~\ref{thm:converse_i}. We work under its hypotheses: $d\ge 2$ is an integer, $L$ is a field of characteristic $p$, and $\alpha,\beta\in L$. As before, we consider the family of polynomials  $f_\l(z)=z^d+\l$  parameterized by $\l\in\Lbar$. We will prove that the set 
\begin{equation}
\label{eq:100}
C(\alpha;\beta)=\left\{\l\in\Lbar\colon\text{ there exists $m\in\N$ such that }f_\l^m(\alpha)=\beta\right\}
\end{equation}
is infinite.

%%%%%%%%%%%%%%%%%%%%%%%%%%%%%%%%%%%%%%%%%%%%%%%%%%%%%%%%%%%%%%%%%%%%%%%%%%%%%%%

\subsection{Proof of Theorem~\ref{thm:converse_i} when the degree is a power of the characteristic of our field}
\label{subsec:d_p_ell}
 
We first prove Theorem~\ref{thm:converse_i} under the additional assumption that $d=p^\ell$ for some $\ell\in\N$.

\begin{lemma}
\label{lem:infinitely_solutions_power_p}
With the above assumptions, $C(\alpha;\beta)$ is an infinite set.
\end{lemma}

\begin{proof} 
Since $d=p^\ell$, the equation $f_\l^m(\alpha)=\beta$ has the following explicit form due to~\eqref{eq:121}:
\begin{equation}
\label{eq:101}
\alpha^{p^{m\ell}}+\sum_{i=0}^{m-1}\l^{p^{i\ell}}=\beta.
\end{equation}
As the equation~\eqref{eq:101} is separable (in $\l$), it has distinct roots for each $m$. Therefore, the set $C(\alpha;\beta)$, being the union of these roots over all $m\in\mathbb{N}$, is infinite. 
\end{proof}

%%%%%%%%%%%%%%%%%%%%%%%%%%%%%%%%%%%%%%%%%%%%%%%%%%%%%%%%%%%%%%%%%%%%%%%%%%%%%%%

\subsection{Conclusion of our proof for  Theorem~\ref{thm:converse_i}}
\label{subsec:d_p_not_ell} In light of Lemma~\ref{lem:infinitely_solutions_power_p}, it suffices to prove Theorem~\ref{thm:converse_i} under the assumption that $d$ is \emph{not} a power of $p$.

We argue by contradiction and assume $C(\alpha;\beta)$ is finite. Let $C(\alpha;\beta)=\{\l_1,\l_2,\dots, \l_r\}$ for some $r\in\N$. For each $m\in\N$, the expansion of $f_\l^m(\alpha)$ from Lemma~\ref{lem:poly} yields the equation: 
\begin{equation}
\label{eq:102}
P_{m,\alpha}(\l)=\l^{d^{m-1}}+\sum_{i=1}^{d^{m-1}-1}c_{m,i}(\alpha)\cdot \l^{d^{m-1}-i} + \alpha^{d^m}=\beta.
\end{equation}
Since $C(\alpha;\beta)=\{\l_1,\l_2,\dots, \l_r\}$, for each $m\in\N$, there exist some nonnegative integers $e_{m,1},\dots, e_{m,r}$ such that the polynomial $P_{m,\alpha}(u)-\beta\in L[u]$ factors as follows in $\Lbar[u]$: 
\begin{equation}
\label{eq:103}
P_{m,\alpha}(u)-\beta=\prod_{i=1}^r (u-\l_i)^{e_{m,i}}.
\end{equation}
Note that the left-hand side (as a polynomial in $L[u]$) is a monic polynomial by~\eqref{eq:102}, which justifies the right-hand side of~\eqref{eq:103} because all the roots of $P_{m,\alpha}(u)-\beta$ are among $\l_1,\dots,\l_r$. On the other hand, we have the recurrence formula:
$$P_{m+1,\alpha}(u)=P_{m,\alpha}(u)^d+u,$$
which combined with equation~\eqref{eq:103} yields
\begin{equation}
\label{eq:104}
\left(\prod_{i=1}^r (u-\l_i)^{e_{m,i}}+\beta \right)^d +u -\beta = \prod_{i=1}^r (u-\l_i)^{e_{m+1,i}}.
\end{equation}
To interpret the equation~\eqref{eq:104}, we consider the subgroup $\Gamma$ of $\mathbb{G}_m^2(L(u))$ spanned by all the elements $(u-\l_i,1)$ and $(1,u-\l_i)$ for $i=1,\dots, r$. We also consider the curve $V$ inside $\mathbb{G}_m^2$ defined by the equation
\begin{equation}
\label{eq:105}
(x+\beta)^d = y + (\beta - u).
\end{equation} 
Since the equation~\eqref{eq:105} for $V$ is linear in $y$, the curve $V$ is geometrically irreducible. Also,  equation~\eqref{eq:104} shows that $V(L(u))\cap\Gamma$ is infinite, as it contains all points of the form 
\begin{equation}
\label{eq:111}
\left(P_{m,\alpha}(u)-\beta,P_{m+1,\alpha}(u)-\beta\right)=\left(\prod_{i=1}^r (u-\l_i)^{e_{m,i}},\prod_{i=1}^r (u-\l_i)^{e_{m+1,i}}\right).
\end{equation}
Note that $\beta^d\ne \beta-u$ because $\beta\in L$ and $u$ is a transcendental variable over $\Lbar$; thus, equation~\eqref{eq:105} shows that the curve $V$ is not the translate of an algebraic subgroup of $\mathbb{G}_m^2$. Indeed, the equation of any translate of a proper subtorus of $\mathbb{G}_m^2$ (defined over $\overline{L(u)}$) is of the form $x^ay^b=\zeta$ for some $\zeta\in\overline{L(u)}^{\ast}$ and some integers $a$ and $b$, not both equal to $0$.

We now employ \cite[Theorem B]{F}, which states that the intersection $V\cap\Gamma$ is a finite union of sets of the form 
\begin{equation}
\label{eq:107}
A(\eta_0,\eta_1,\dots,\eta_s;k_1,\dots, k_s)\cdot H,
\end{equation}
where $H\subseteq \Gamma$ is a subgroup, $s\in\N$, while for some given $\eta_0,\eta_1,\dots, \eta_s\in \mathbb{G}_m^2(\overline{L(u)})$ and some $k_1,\dots, k_s\in\N$, we have that 
\begin{equation}
\label{eq:F-set}
A(\eta_0,\eta_1,\dots, \eta_s,k_1,\dots, k_s)\colonequals\left\{\eta_0\cdot \prod_{i=1}^s \eta_i^{p^{k_in_i}}\colon n_i\in \N\right\}.
\end{equation}
In formula~\eqref{eq:107}, we use the notation $C_1\cdot C_2$ for any two subsets $C_1,C_2\subset \mathbb{G}_m^2$ to denote the set $\left\{c_1\cdot c_2\colon c_i\in C_i\text{ for }i=1,2\right\}$. 

Furthermore, as explained in \cite[Remark~2.11]{F}, there exists some $N\in\N$ such that 
\begin{equation}
\label{eq:F-set-2}
\eta_i^N\in\Gamma\text{ for }i=0,1,\dots, s.
\end{equation}
Since $V$ is an irreducible curve which is not a translate of an algebraic subgroup of $\mathbb{G}_m^2$, the subgroups $H$ from equation~\eqref{eq:107} must be finite (see also \cite[Corollary~2.3]{TAMS}). Therefore, at the expense of replacing each set~\eqref{eq:107} by finitely many other sets of the form~\eqref{eq:107}, we may assume from now on that $H$ is the trivial subgroup of $\Gamma$.

As $V$ is a curve, we have $s=1$ in equations~\eqref{eq:107}~and~\eqref{eq:F-set} by \cite[Corollary~2.3]{TAMS}. Thus, the intersection $V\cap\Gamma$ is a union of finitely many sets of the form
$$A(\eta_0,\eta_1;k)=\left\{\eta_0\cdot \eta_1^{p^{kn}}\colon n\in\N \right\},$$
for some given $\eta_0,\eta_1$ satisfying \eqref{eq:F-set-2} and some  $k\in\N$. Next, we write each $\eta_j\colonequals(\gamma_{j,1},\gamma_{j,2})$ for $j=0,1$. Using \eqref{eq:F-set-2}, we can express 
\begin{equation}
\label{eq:108}
\gamma_{j,1}^N\equalscolon\prod_{i=1}^r (u-\l_i)^{a_{j,i}}, 
\end{equation}
for some integers $a_{j,i}$ for $j=0,1$ and $i=1,\dots, r$. Therefore, the $N$-th power of the first components (in $\mathbb{G}_m^2$) of the points in $A(\eta_0,\eta_1;k)$ are of the form
\begin{equation}
\label{eq:109}
\prod_{i=1}^r \left(u-\l_i\right)^{a_{0,i}+ a_{1,i}p^{kn}},
\end{equation}
as we vary $n$ in $\N$. In particular, the degrees in $u$ of the polynomials appearing in equation~\eqref{eq:109} form the set 
\begin{equation}
\label{eq:110}
\left\{A_0+A_1p^{kn}\colon n\in\N\right\},
\end{equation}
where $A_0\colonequals\sum_{i=1}^r a_{0,i}$ and $A_1\colonequals\sum_{i=1}^r a_{1,i}$. On the other hand, we already know that $V\cap\Gamma$ contains the elements from equation~\eqref{eq:111}; in particular, each $P_{m,\alpha}(u)-\beta$ (as we vary $m\in\N$) appears as the first component of an element from the intersection $V\cap\Gamma$. We have (see \eqref{eq:102}) that the degree (in $u$) of $P_{m,\alpha}(u)-\beta$ is $d^{m-1}$. Therefore, the set 
$$\left\{d^{m-1}\colon m\in\N\right\}$$
is contained in the union of finitely many sets of the form \eqref{eq:110}. In particular, for some choice of $A_0$ and $A_1$ (and $k\in\N$), there exist infinitely many $m\in\N$ such that the equation
\begin{equation}
\label{eq:112}
A_0+A_1p^{kn}=d^{m-1},
\end{equation}
has some solution $n\in\mathbb{N}$. However, because $d$ is \emph{not} a power of $p$, equation~\eqref{eq:112} can have only finitely many solutions. This finiteness is a consequence of deep results in Diophantine analysis; for instance, it is a very special case of Laurent's famous theorem on the Mordell-Lang conjecture for algebraic tori \emph{in characteristic $0$} \cite[Th\'eor\`eme~2]{Laurent}. Alternatively, the same conclusion follows from the theory of the $S$-unit equation \cite[Theorem~1.1]{Schlickewei}. This final contradiction shows that the assumption that the equations~\eqref{eq:102} (as we vary $m\in\mathbb{N}$) have only finitely many roots $\l\in\Lbar$ is untenable. Therefore, the set $C(\alpha;\beta)$ must be infinite. 

This concludes our proof of  Theorem~\ref{thm:converse_i}.

%%%%%%%%%%%%%%%%%%%%%%%%%%%%%%%%%%%%%%%%%%%%%%%%%%%%%%%%%%%%%%%%%%%%%%%%%%%%%%%
%%%%%%%%%%%%%%%%%%%%%%%%%%%%%%%%%%%%%%%%%%%%%%%%%%%%%%%%%%%%%%%%%%%%%%%%%%%%%%%

\section{Proof of Theorem~\ref{thm:converse_ii}}
\label{sec:ii}

In this Section, we prove Theorem~\ref{thm:converse_ii}. We work under its stated hypotheses: $L$ is a field of characteristic $p$ with points $\alpha_1,\alpha_2,\beta\in L$ satisfying $\alpha_1\ne\alpha_2$, and $d=p^\ell$ for some $\ell\in\N$. As before, we let $f_\l(z)=z^d+\l$ for each $\l\in\Lbar$ and consider the set:
$$C(\alpha_1,\alpha_2;\beta)=\left\{\l\in\Lbar\colon\text{ there exist $m,n\in\N$ such that $f_\l^m(\alpha_1)=f_\l^n(\alpha_2)=\beta$}\right\}.$$
We let $\delta_1\colonequals\alpha_2-\alpha_1$ and $\delta_2\colonequals\beta-\alpha_1$. Furthermore, we assume 
\begin{equation}
\label{eq:79}
\delta_1\in\Fq^\ast\text{ and }\delta_2\in\Fq 
\end{equation}
for some finite subfield $\Fq\subseteq L$. Our goal is to prove that $C(\alpha_1,\alpha_2;\beta)$ is infinite if  the system of two equations:
\begin{equation}
\label{eq:system_1}
\left\{\begin{array}{ccc}
\delta_1 & = & \sum_{i=0}^{s_1-1}\gamma^{p^{ik\ell}}\\
\\
\delta_2 & = & \sum_{i=0}^{s_2-1}\gamma^{p^{ik\ell}}\end{array}\right.
\end{equation}
has a solution  $(\gamma,k,s_1,s_2)\in\Fq^\ast\times\N\times\N\times\N$. Moreover, we will also show that if the system~\eqref{eq:system_1} has no such solution, then $C(\alpha_1,\alpha_2;\beta)$ is \emph{empty}. 
We split our proof of Theorem~\ref{thm:converse_ii} over several Subsections of Section~\ref{sec:ii}. 

%%%%%%%%%%%%%%%%%%%%%%%%%%%%%%%%%%%%%%%%%%%%%%%%%%%%%%%%%%%%%%%%%%%%%%%%%%%%%%%

\subsection{Strategy for proving Theorem~\ref{thm:converse_ii}}

We obtain the desired conclusion from Theorem~\ref{thm:converse_ii} by first finding explicit conditions which are equivalent with the existence of \emph{at least one} $\l\in C(\alpha_1,\alpha_2;\beta)$ (see Proposition~\ref{prop:findings}); more precisely, the existence of some $\l\in C(\alpha_1,\alpha_2;\beta)$ is equivalent to a solution to the system~\eqref{eq:system_1}. Then we prove that one solution to the system~\eqref{eq:system_1} leads to infinitely many solutions to the system~\eqref{eq:system_1} and in turn, this leads to  infinitely many $\l\in C(\alpha_1,\alpha_2;\beta)$ (see Proposition~\ref{prop:one_to_infinite}).  

We will also prove that equation~\eqref{eq:79} alone \emph{does not always} imply that the set $C(\alpha_1,\alpha_2;\beta)$ is infinite (see Proposition~\ref{prop:not_always}); in other words, there are examples of $\delta_1,\delta_2$ as in~\eqref{eq:79} such that the system~\eqref{eq:system_1} has \emph{no solutions} and therefore, $C(\alpha_1,\alpha_2;\beta)$ is \emph{empty}.

%%%%%%%%%%%%%%%%%%%%%%%%%%%%%%%%%%%%%%%%%%%%%%%%%%%%%%%%%%%%%%%%%%%%%%%%%%%%%%%

\subsection{Reductions in our proof of Theorem~\ref{thm:converse_ii}}

Since $\alpha_2-\alpha_1,\beta-\alpha_1\in\Fpbar$, then $\alpha_1,\alpha_2,\beta\in \Fpbar(\alpha_1)$; so, without loss of generality (see also Lemma~\ref{lem:trdeg_lambda}), we may assume that $L=\Fpbar(\alpha_1)$.

We first obtain more precise information regarding $\delta_1=\alpha_2-\alpha_1$ and $\delta_2=\beta-\alpha_1$ under the assumption that $C(\alpha_1,\alpha_2;\beta)$ is \emph{nonempty}. To that end, let $\l\in C(\alpha_1,\alpha_2;\beta)$ and let $m,n\in\N$ such that 
\begin{equation}
\label{eq:81}
f_\l^m(\alpha_1)=f_\l^n(\alpha_2)=\beta.
\end{equation}
We will see next that we can, in fact, \emph{always} assume $m>n$ in \eqref{eq:81}.

\begin{lemma}
\label{lem:actually}
With the notation as in \eqref{eq:81}, we can assume that $m>n$.
\end{lemma}

\begin{proof}
First, we note that by Lemma~\ref{lem:easy_perm}, the case $m=n$ implies $\alpha_1=\alpha_2$, which contradicts the hypothesis in Theorem~\ref{thm:converse_ii}. Thus, it suffices to show that if $n>m$ in \eqref{eq:81}, then we can replace $m$ by a suitable integer larger than $n$ so that \eqref{eq:81} holds.

Assume $n>m$ in \eqref{eq:81}. Since $f_\l$ induces a permutation on $\Lbar$, equation~\eqref{eq:81} implies:
\begin{equation}
\label{eq:810}
f_\l^{n-m}(\alpha_2)=\alpha_1.
\end{equation}
Since $\alpha_1-\alpha_2=-\delta_1\in\Fpbar$, we can apply Lemma~\ref{lem:prep_0} to equation~\eqref{eq:810} to deduce that both $\alpha_1$ and $\alpha_2$ are periodic points for $f_\l$. Let $t_0\in\N$ be the period of $\alpha_1$ under the action of $f_\l$. Define $m'\colonequals m+nt_0$. By periodicity of $\alpha_1$, we have $f_\l^{m'}(\alpha_1)=f_\l^m(\alpha_1)=\beta$. This gives us another instance of \eqref{eq:81}, namely, $f_\l^{m'}(\alpha_1)=f_\l^n(\alpha_2)=\beta$. Since $m'>n$, this concludes our proof of Lemma~\ref{lem:actually}.
\end{proof}

%%%%%%%%%%%%%%%%%%%%%%%%%%%%%%%%%%%%%%%%%%%%%%%%%%%%%%%%%%%%%%%%%%%%%%%%%%%%%%%

\subsection{Different points in the orbit which differ by an element from $\Fpbar$}

According to Lemma~\ref{lem:actually}, we may assume that $m>n$ in equation~\eqref{eq:81}.

Because $d=p^\ell$, we recall from \eqref{eq:121} that for every $n\in\N$, 
\begin{equation}
\label{eq:84}
f_\l^n(z)=z^{p^{n\ell}}+\sum_{i=0}^{n-1}\l^{p^{i\ell}}.
\end{equation}
In particular, we have (for each $z$ and $\epsilon$)
\begin{equation}
\label{eq:85}
f_\l^n(z+\epsilon)=z^{p^{n\ell}}+\epsilon^{p^{n\ell}}+ \sum_{i=0}^{n-1}\l^{p^{i\ell}}=f_\l^n(z)+\epsilon^{p^{n\ell}}.
\end{equation}
Since $f_\l(z)=z^d+\l=z^{p^\ell}+\l$ induces a permutation on $\Lbar$ and $m>n$, equation~\eqref{eq:81} yields that $f_\l^{m-n}(\alpha_1)=\alpha_2$. We recall that 
\begin{equation}
\label{eq:82}
\text{$\delta_1\colonequals\alpha_2-\alpha_1\in\Fq^\ast$ and $\delta_2\colonequals\beta-\alpha_1\in\Fq$; also, we let}
\end{equation}
\begin{equation}
\label{eq:83}
\text{$k\in\N$ be minimal with the property that $f_\l^k(\alpha_1)-\alpha_1\in\Fq$.}
\end{equation}
Note that $f_\l^m(\alpha_1)=\beta=\alpha_1+\delta_2$ and $f_\l^{m-n}(\alpha_1)=\alpha_2=\alpha_1+\delta_1$; so, equation~\eqref{eq:82} ensures that $k$ from~\eqref{eq:83} is well-defined.  We let $\gamma\colonequals f_\l^k(\alpha_1)-\alpha_1$; due to~\eqref{eq:83}, we have
\begin{equation}
\label{eq:830}
\gamma\in\Fq.
\end{equation}
For each $a\ge 0$, we write $u_a \colonequals f_\l^{ka}(\alpha_1)-\alpha_1$; clearly, $u_0=0$ and $u_1=\gamma$. A simple induction on $a$ using equation~\eqref{eq:85} establishes the recurrence relation
\begin{equation}
\label{eq:86}
u_{a+1}=u_a^{p^{k\ell}}+\gamma.
\end{equation}
To show the inductive step, we compute $u_{a+1}=f_\l^{k(a+1)}(\alpha_1)-\alpha_1$ and so,
$$u_{a+1}=f_\l^k\left(f_\l^{ka}(\alpha_1)\right)-\alpha_1= f_\l^k\left(\alpha_1+u_a\right)-\alpha_1\overset{\eqref{eq:85}}{=} f_\l^k(\alpha_1)+u_a^{p^{k\ell}} -\alpha_1.$$
Since $f_\l^k(\alpha_1)-\alpha_1=\gamma$, we obtain the recurrence relation~\eqref{eq:86}. 
In particular, equation~\eqref{eq:86} shows that 
\begin{equation}
\label{eq:87}
f_\l^{ka}(\alpha_1)-\alpha_1=\sum_{i=0}^{a-1}\gamma^{p^{ik\ell}}\in\Fq.
\end{equation}

\begin{lemma}
\label{lem:multiples_l}
Let $s\in\N$ such that $f_\l^s(\alpha_1)-\alpha_1\in\Fq$. Then $k\mid s$.
\end{lemma}

\begin{proof}
If $k\nmid s$, then there exists $a\ge 0$ and $r\in\{1,\dots, k-1\}$ such that $s=ka+r$. Then equation~\eqref{eq:87} yields that $u_a=f_\l^{ka}(\alpha_1)-\alpha_1\in\Fq$. Using equation~\eqref{eq:85}, we compute:
\begin{equation}
\label{eq:88}
f^s_\l(\alpha_1)=f^{ka+r}_\l(\alpha_1)=f^r_\l\left(f_\l^{ka}(\alpha_1)\right)= f^r_\l(\alpha_1+u_a)=f_\l^r(\alpha_1)+u_a^{p^{r\ell}}.
\end{equation}
Since $u_a\in\Fq$ and $f^s_\l(\alpha_1)-\alpha_1\in\Fq$ (our hypothesis), we deduce $f_\l^r(\alpha_1)-\alpha_1\in\Fq$ from equation~\eqref{eq:88}. However, since $1\le r\le k-1$, this contradicts the minimality of $k$ from~\eqref{eq:83}. Therefore, we must have that $k\mid s$, as desired.
\end{proof}

Using Lemma~\ref{lem:multiples_l} along with equation~\eqref{eq:82}, we conclude that $m-n=ks_1$ and $m=ks_2$ for some positive integers $s_1<s_2$.

%%%%%%%%%%%%%%%%%%%%%%%%%%%%%%%%%%%%%%%%%%%%%%%%%%%%%%%%%%%%%%%%%%%%%%%%%%%%%%%

\subsection{The defining system of two equations and one unknown} 
\label{subsec:system}

According to equations~\eqref{eq:82}~and~\eqref{eq:87}, the original condition from~\eqref{eq:81} translates to \emph{two} equations: 
$$\delta_1=f_\l^{m-n}(\alpha_1)-\alpha_1=f_\l^{ks_1}(\alpha_1)-\alpha_1 =u_{ks_1}\text{ and so,}$$
\begin{equation}
\label{eq:89}
\delta_1= \sum_{i=0}^{s_1-1}\gamma^{p^{ik\ell}}\text{; and}
\end{equation}
$$\delta_2=f_\l^{m}(\alpha_1)-\alpha_1=f_\l^{ks_2}(\alpha_1)-\alpha_1 =u_{ks_2}\text{ and so,}$$
\begin{equation}
\label{eq:90}
\delta_2=\sum_{i=0}^{s_2-1}\gamma^{p^{ik\ell}}.
\end{equation}
We summarize our findings so far in the following Proposition.
\begin{proposition}
\label{prop:findings}
With the notation as in Theorem~\ref{thm:converse_ii}, assume $\delta_1=\alpha_2-\alpha_1\ne 0$ and $\delta_2=\beta-\alpha_1$ are contained in $\Fq$. Then the set $C(\alpha_1,\alpha_2;\beta)$ is nonempty if and only if the system~\eqref{eq:system_1}  has a solution $(\gamma,k,s_1,s_2)\in\Fq^\ast\times\N\times\N\times\N$.
\end{proposition}

\begin{proof}
First, assume $C(\alpha_1, \alpha_2; \beta)$ is nonempty and let $\lambda\in C(\alpha_1, \alpha_2; \beta)$. Then $f_\l^m(\alpha_1)=f_\l^n(\alpha_2)=\beta$ for some $m, n\in\mathbb{N}$. Since $\alpha_1\neq \alpha_2$,  Lemma~\ref{lem:easy_perm} implies $m\ne n$, and by Lemma~\ref{lem:actually}, we may assume $m>n$. This leads to the system~\eqref{eq:system_1} (see also equations~\eqref{eq:89}~and~\eqref{eq:90}) to have a solution $(\gamma,k,s_1,s_2)\in\Fq\times\N\times\N\times\N$. Furthermore, the assumption $\delta_1\ne 0$ ensures that $\gamma\ne 0$ due to equation~\eqref{eq:89}. 

Now, for the converse implication, we assume the system~\eqref{eq:system_1} has a solution $(\gamma,k,s_1,s_2)\in\Fq^\ast\times\N\times\N\times\N$. From equation~\eqref{eq:84} we solve for $\l\in\Lbar$ such that
\begin{equation}
\label{eq:800}
f_\l^k(\alpha_1)=\alpha_1+\gamma.
\end{equation}
Using again \eqref{eq:84} coupled with \eqref{eq:800}, along with equations~\eqref{eq:89}~and~\eqref{eq:90}, we obtain 
\begin{equation}
\label{eq:801}
f_\l^{ks_1}(\alpha_1)=\alpha_1+\delta_1=\alpha_2\text{ and }
\end{equation}  
\begin{equation}
\label{eq:802}
f_\l^{ks_2}(\alpha_1)=\alpha_1+\delta_2=\beta.
\end{equation}
So, choosing $m\colonequals ks_2$ and $n \colonequals k(s_2-s_1)$, we get that equation~\eqref{eq:81} holds; therefore, $\l\in C(\alpha_1,\alpha_2;\beta)$, as desired.

This concludes our proof of Proposition~\ref{prop:findings}.
\end{proof}

%%%%%%%%%%%%%%%%%%%%%%%%%%%%%%%%%%%%%%%%%%%%%%%%%%%%%%%%%%%%%%%%%%%%%%%%%%%%%%%

\subsection{One solution to our system generates infinitely many parameters}

We continue with the notation as in Subsection~\ref{subsec:system}.  
The next Proposition shows that once there exists \emph{one} solution $\gamma\in\Fq^\ast$ (and $k,s_1,s_2\in\Fq$) to the system of equations~\eqref{eq:89}~and~\eqref{eq:90}, then  $C(\alpha_1,\alpha_2;\beta)$ is infinite.

\begin{proposition}
\label{prop:one_to_infinite}
If there exists a solution $(\gamma,k,s_1,s_2)\in\Fq^\ast\times \N\times\N\times\N$ to the equations~\eqref{eq:89}~and~\eqref{eq:90}, then $C(\alpha_1,\alpha_2;\beta)$ is infinite.
\end{proposition}

\begin{proof}
We write $q\colonequals p^r$ for some $r\in\N$. We note that once we have a solution $(\gamma,k,s_1,s_2)$ to the system~\eqref{eq:89}-\eqref{eq:90}, $(\gamma,k+r,s_1,s_2)$ also solves the above system. However, this new solution to the system leads to a different element in $C(\alpha_1,\alpha_2;\beta)$. Indeed, for any such solution $(\gamma,k,s_1,s_2)$, the corresponding parameter $\l\in C(\alpha_1,\alpha_2;\beta)$ satisfies
\begin{equation}
\label{eq:91}
f_\l^k(\alpha_1)=\alpha_1+\gamma\text{, }f_\l^{ks_1}(\alpha_1)=f_\l^{m-n}(\alpha_1)=\alpha_1+\delta_1\text{ and }f_\l^{ks_2}(\alpha_1)=f_\l^m(\alpha_1)=\alpha_1+\delta_2.
\end{equation}   
Combined with~\eqref{eq:84}, the first equation from \eqref{eq:91} yields that
\begin{equation}
\label{eq:92}
\alpha_1^{p^{k\ell}}+\sum_{i=0}^{k-1}\l^{p^{i\ell}}=\alpha_1+\gamma.
\end{equation}
Equation~\eqref{eq:92} in $\l$ is a separable equation of degree $p^{(k-1)\ell}$ and hence has distinct solutions. So, increasing $k$ leads to additional solutions to the new equation~\eqref{eq:92}. Therefore, we have infinitely many elements in $C(\alpha_1,\alpha_2;\beta)$ simply assuming the existence of one solution $(\gamma,k,s_1,s_2)$ to the system of equations~\eqref{eq:89}~and~\eqref{eq:90}.
\end{proof}

%%%%%%%%%%%%%%%%%%%%%%%%%%%%%%%%%%%%%%%%%%%%%%%%%%%%%%%%%%%%%%%%%%%%%%%%%%%%%%%

\subsection{Conclusion of our proof for Theorem~\ref{thm:converse_ii}}

If the system~\eqref{eq:system_1} has no solutions, then $C(\alpha_1,\alpha_2;\beta)$ must be empty by Proposition~\ref{prop:findings}. Now, if the system~\eqref{eq:system_1} has a solution, then Proposition~\ref{prop:one_to_infinite} yields that there are actually infinitely many $\l\in C(\alpha_1,\alpha_2;\beta)$. 

This concludes our proof of Theorem~\ref{thm:converse_ii}.

%%%%%%%%%%%%%%%%%%%%%%%%%%%%%%%%%%%%%%%%%%%%%%%%%%%%%%%%%%%%%%%%%%%%%%%%%%%%%%%

\subsection{Sometimes the set $C(\alpha_1,\alpha_2;\beta)$ is empty}

Next, we show that condition~\eqref{eq:79} alone from Theorem~\ref{thm:converse_ii} does \emph{not} always guarantee the existence of infinitely many $\l\in C(\alpha_1,\alpha_2;\beta)$; in other words, there are instances when the  system~\eqref{eq:system_1} is not solvable (despite the fact that $\alpha_2-\alpha_1,\beta-\alpha_1\in\Fpbar$) and therefore,  $C(\alpha_1,\alpha_2;\beta)$ is empty.

\begin{proposition}
\label{prop:not_always}
Let $d=p^\ell$ and let $\alpha_1,\alpha_2,\beta\in L$. If $\delta_1=\alpha_2-\alpha_1$ and $\delta_2=\beta-\alpha_1$ simultaneously satisfy the following conditions:
\begin{itemize}
\item[(1)] $\delta_1,\delta_2\in\mathbb{F}_{p^2}\setminus\Fp$, 
\item[(2)] $\delta_1-\delta_2\notin\Fp$, and
\item[(3)] $\frac{\delta_1}{\delta_2}\notin\Fp$,
\end{itemize}
then $C(\alpha_1,\alpha_2;\beta)$ is empty.
\end{proposition}

\begin{remark}
\label{rem:no_solutions}
For each prime $p>2$, one can construct examples where Proposition~\ref{prop:not_always} applies. For instance, choosing $(\delta_1,\delta_2)=(\epsilon, 2\epsilon +1)$ for any $\epsilon\in\mathbb{F}_{p^2}\setminus\Fp$ satisfies the conditions~(1)-(3) in Proposition~\ref{prop:not_always}. Therefore, the corresponding set $C(\alpha_1,\alpha_2;\beta)$ is empty, even though the points $\alpha_1,\alpha_2,\beta$ satisfy the alternative~(ii) from the conclusion of Theorem~\ref{thm:main}. 
\end{remark}

\begin{proof}[Proof of Proposition~\ref{prop:not_always}.]
We argue by contradiction and assume $C(\alpha_1, \alpha_2; \beta)\neq \emptyset$. Therefore, there exists $\l\in\Lbar$ and  $m,n\in\N$ such that $f_\l^m(\alpha_1)=f_\l^n(\alpha_2)=\beta$. As shown in Lemma~\ref{lem:actually}, we may assume that $m>n$. By Proposition~\ref{prop:findings}, this triple $(\l,m,n)\in\Lbar\times\N\times\N$ leads to a solution $(\gamma,k,s_1,s_2)\in\mathbb{F}_{p^2}\times\N\times\N\times\N$ to the system of equations~\eqref{eq:89}~and~\eqref{eq:90}. Note that, due to~\eqref{eq:830}, we can assume $\gamma\in\mathbb{F}_{p^2}$ because $\delta_1,\delta_2\in\mathbb{F}_{p^2}$. We will show that there are no solutions $(\gamma,k,s_1,s_2)\in\mathbb{F}_{p^2}\times\N\times\N\times\N$ to the system~\eqref{eq:89}-\eqref{eq:90}. The proof is divided into two cases based on the parity of $k\cdot \ell$.

\begin{lemma}
\label{lem:even_no}
With the above notation, there are no solutions to the system~\eqref{eq:89}-\eqref{eq:90} if $k\cdot \ell$ is even.
\end{lemma}

\begin{proof}[Proof of Lemma~\ref{lem:even_no}.]
Since any sought solution $\gamma$ of the system~\eqref{eq:89}-\eqref{eq:90} lives in $\mathbb{F}_{p^2}$, we have $\gamma^{p^{ik\ell}}=\gamma$ for each $i\ge 0$. So, the system~\eqref{eq:89}-\eqref{eq:90} simplifies to
\begin{equation}
\label{eq:93}
s_1\cdot \gamma=\delta_1\text{ and }s_2\cdot \gamma=\delta_2.
\end{equation}
However, \eqref{eq:93} implies $\delta_1/\delta_2=s_1/s_2$, which contradicts condition~(3) from the hypotheses of Proposition~\ref{prop:not_always}. Thus, no solutions exist when $k\ell$ is even. 
\end{proof}

\begin{lemma}
\label{lem:odd_no}
With the above notation, there are no solutions to the system~\eqref{eq:89}-\eqref{eq:90} if $k\cdot \ell$ is odd.
\end{lemma}

\begin{proof}[Proof of Lemma~\ref{lem:odd_no}.]
In this case, we know that for any even $s\in\N$, we have that 
\begin{equation}
\label{eq:94}
\sum_{i=0}^{s-1} \gamma^{p^{ik\ell}}= \frac{s}{2}\cdot {\rm Tr}_{\mathbb{F}_{p^2}/\Fp}(\gamma),
\end{equation}
while for any odd $s\in\N$, we have that
\begin{equation}
\label{eq:95}
\sum_{i=0}^{s-1}\gamma^{p^{ik\ell}}=\gamma + \frac{s-1}{2}\cdot {\rm Tr}_{\mathbb{F}_{p^2}/\Fp}(\gamma).
\end{equation}
So, if $s_j$ is  even for some $j=1,2$, then \eqref{eq:94} shows that the system~\eqref{eq:89}-\eqref{eq:90} leads to $\delta_j\in\Fp$, which contradicts condition~(1) from the hypotheses of Proposition~\ref{prop:not_always}. If both $s_1$ and $s_2$ are odd, then \eqref{eq:94} yields that $\delta_1-\delta_2\in\Fp$, which contradicts condition~(2) from the hypotheses of Proposition~\ref{prop:not_always}. In all cases, we reach a contradiction, so no solution to the system~\eqref{eq:89}-\eqref{eq:90} can exist if $k\ell$ is odd.
\end{proof}

Combining Lemmas~\ref{lem:even_no} and \ref{lem:odd_no} concludes our proof of Proposition~\ref{prop:not_always}.
\end{proof}

%%%%%%%%%%%%%%%%%%%%%%%%%%%%%%%%%%%%%%%%%%%%%%%%%%%%%%%%%%%%%%%%%%%%%%%%%%%%%%%
%%%%%%%%%%%%%%%%%%%%%%%%%%%%%%%%%%%%%%%%%%%%%%%%%%%%%%%%%%%%%%%%%%%%%%%%%%%%%%%

\section{The case when all the points live in a finite field}
\label{subsec:fpbar}

Throughout this Section, we work with $\alpha_1,\alpha_2,\beta\in\Fpbar$ and the family of polynomials $f_\l(z)\colonequals z^d+\l$ parameterized by $\l\in\Fpbar$. Our goal is to determine whether the set
\begin{equation}
\label{eq:010}
C(\alpha_1,\alpha_2;\beta)=\left\{\l\in\Fpbar\colon\text{ there exist $m,n\in\N$ such that $f_\l^m(\alpha_1)=f_\l^n(\alpha_2)=\beta$}\right\}.
\end{equation}
is infinite. Note that restricting $\lambda$ to $\Fpbar$ is sufficient, as Lemma~\ref{lem:trdeg_lambda} implies that if $f_\l^m(\alpha_1)=\beta$, then we have that $\l\in \overline{\Fp(\alpha_1,\beta)}=\Fpbar$ because $\alpha_1, \beta\in\Fpbar$. 

Since Theorem~\ref{thm:converse_ii} provides an explicit answer when $d=p^\ell$ (everything depends on the existence of a solution to a system~\eqref{eq:system_010} of equations), we assume that $d$ is \emph{not} a power of $p$. Under this hypothesis, as long as $\alpha_1^d\ne \alpha_2^d$, there is no visible dynamical relation \emph{globally} between $\alpha_1$ and $\alpha_2$ with respect to the \emph{entire} family of polynomials $f_\l(z)$. So, drawing on the intuition from Theorem~\ref{thm:main_0}, one would expect that if $\alpha_1^d\ne \alpha_2^d$, then $C(\alpha_1,\alpha_2;\beta)$ is finite. However, based on extensive computation (of more than $100$ examples), we believe the \emph{opposite} is true.

\begin{conjecture}
\label{conf:fpbar}
Let $d\ge 2$ be an integer which is not a power of $p$ and let  $\alpha_1,\alpha_2,\beta\in \Fpbar$. For any $\l\in\Fpbar$, let $f_\l(z)\colonequals z^d+\l$.  Then the set $C(\alpha_1,\alpha_2;\beta)$ (see \eqref{eq:010}) is infinite.
\end{conjecture}

This Section is organized as follows. In Subsection~\ref{subsec:9.1}, we describe the algorithm used for testing various cases in Conjecture~\ref{conf:fpbar}. In Subsection~\ref{subsec:9.2}, we formulate Conjecture~\ref{conf:fpbar:perm-poly}, which is a refinement of Conjecture~\ref{conf:fpbar} in the special case $\alpha_1=\beta$. We gathered in Subsection~\ref{subsec:9.3} some of the many examples we tested, all of which support both Conjectures~\ref{conf:fpbar}~and~\ref{conf:fpbar:perm-poly}. In Subsection~\ref{subsec:9.4}, we present an example for a different family of polynomials, which in turn suggests Question~\ref{quest:general}. Finally, in Subsection~\ref{subsec:9.5}, we conclude our paper with a brief discussion about the collision of multiple orbits and formulate Question~\ref{quest:multiple}. 

%%%%%%%%%%%%%%%%%%%%%%%%%%%%%%%%%%%%%%%%%%%%%%%%%%%%%%%%%%%%%%%%%%%%%%%%%%%%%%%

\subsection{Algorithm for testing Conjecture~\ref{conf:fpbar}}
\label{subsec:9.1}

Since we are now studying the case where $\alpha_1, \alpha_2, \beta\in\Fpbar$, we fix a sufficiently large $q$ such that $\alpha_1, \alpha_2, \beta\in\mathbb{F}_q$. For a fixed $\alpha$ and $n\in\mathbb{N}$, recall that $P_{n, \alpha}(\lambda)\colonequals f_{\lambda}^{m}(\alpha)$ is a polynomial in $\lambda$ with degree $d^{n-1}$. By definition, $f^{m}_{\lambda}(\alpha_1)=\beta$ if and only if $\lambda$ is a root of $P_{m, \alpha_1}(\lambda)-\beta=0$.  Directly finding the roots of the polynomials $P_{m, \alpha_1}(\lambda)-\beta$ and $P_{n, \alpha_2}(\lambda)-\beta$ is often infeasible due to their high degrees. We now describe a more efficient algorithm for finding values of $\lambda$ in $C(\alpha_1, \alpha_2; \beta)$ by searching for common factors of these polynomials over finite fields.

The algorithm proceeds by checking, for each irreducible polynomial $g(\lambda)$ over a finite field $\mathbb{F}_q$, whether it divides $P_{n, \alpha}(\lambda)-\beta$ for some $n\in\mathbb{N}$. This check is performed efficiently in the quotient ring $\mathbb{F}_q[\lambda]/\langle g(\lambda)\rangle$. For a given $\alpha, \beta\in\mathbb{F}_q$ and an irreducible polynomial $g(\lambda)\in\mathbb{F}_q[\lambda]$, the procedure is as follows: 

\begin{enumerate}
    \item We recursively compute the sequence of remainders $r_n \colonequals f_{\lambda}^{n}(\alpha) \pmod{g(\lambda)}$ by setting $r_0=\alpha$ and computing $r_n = r_{n-1}^{d}+\lambda$ in $\mathbb{F}_q[\lambda]/\langle g(\lambda)\rangle$ for $n\geq 1$.
    \item If $r_n=\beta$ for some $n\in\mathbb{N}$, then $g(\lambda)$ divides $P_{n, \alpha}(\lambda)-\beta$ and the algorithm stops for this particular polynomial $g(\l)$. Consequently, every root of $g(\lambda)=0$ is a solution to $f_{\lambda}^{n}(\alpha)=\beta$. 
    \item If we encounter a remainder $r_n$ that has appeared previously (i.e., $r_n=r_i$ for some $i<n$) and \emph{none} of the remainders $r_1,\dots, r_n$ are   equal to $\beta$, a cycle has been detected. We conclude that $\beta$ will not be reached, and we terminate the search for this polynomial $g(\lambda)$. Also, as an aside, the pair $(i,n)$ (where $i$ and $n$ are minimal with the property that $f_\l^i(\alpha)=f_\l^n(\alpha)$) is called the \emph{preperiodicity portrait} for $\alpha$ (under the action of $f_\l$). 
\end{enumerate}

The length of the sequence before repetition is bounded by $q^{\deg(g)}$. For practical implementation, we must set a maximum number of iterations. If this threshold is reached without finding $\beta$ or a cycle, the algorithm terminates inconclusively for $g(\lambda)$.

To find elements in $C(\alpha_1, \alpha_2, \beta)$, we apply this procedure to both $\alpha_1$ and $\alpha_2$. For each monic irreducible polynomial $g(\lambda)\in\mathbb{F}_q[\lambda]$, we search for an integer $m$ such that $g(\lambda)$  divides $P_{m, \alpha_1}(\lambda)-\beta$. If successful, we then search for an integer $n$ such that $g(\lambda)$ also divides $P_{n, \alpha_2}(\lambda)-\beta$. If both searches succeed, all roots of $g(\lambda)$ belong to the set $C(\alpha_1, \alpha_2; \beta)$. 

\subsection{The case where $\alpha_1=\beta$} 
\label{subsec:9.2}

We describe a special case in which permutation polynomials naturally arise. Suppose $\alpha_1=\beta\in\mathbb{F}_q$. We are searching for $\lambda\in\overline{\mathbb{F}_p}$ such that for some $m, n\in\mathbb{N}$, the following equalities hold:
\begin{equation*}
f_{\lambda}^{m}(\alpha_1)=\alpha_1 \quad \text{ and } \quad f_{\lambda}^{n}(\alpha_2)=\alpha_1. 
\end{equation*}
We now explain how first condition, $f^{m}_{\lambda}(\alpha_1)=\alpha_1$, is automatically satisfied for certain choices of $\lambda$. Suppose the polynomial $P_{n, \alpha_2}(\lambda)-\alpha_1$ has an irreducible factor $h(\lambda)$ of degree $k$ over $\mathbb{F}_q$. Let $\lambda_0\in\mathbb{F}_{q^k}$ be any root of $h(\lambda)$. If $\gcd(d, q^k-1)=1$, the map $z\mapsto z^d$ is a permutation of $\mathbb{F}_{q^k}$, which implies that $f_{\lambda_0}(z)=z^d+\lambda_0$ is also a permutation of $\mathbb{F}_{q^k}$. Since $\alpha_1\in\mathbb{F}_q\subseteq \mathbb{F}_{q^{k}}$, then  $\alpha_1$ must be periodic under $f_\l$; thus, $f_{\lambda_0}^{m}(\alpha_1)=\alpha_1$ for some $m\in\mathbb{N}$. The argument shows that when $\gcd(d, q^{k}-1)=1$, each root $\lambda_0\in\mathbb{F}_{q^k}$ of $h(\lambda)$ satisfies $f^m_{\lambda_0}(\alpha_1)=f^n_{\lambda_0}(\alpha_2)=\alpha_1$ for some $m, n\in\mathbb{N}$. Therefore, to prove that the set $C(\alpha_1, \alpha_2; \alpha_1)$ is infinite, it suffices to show that the polynomials $P_{n, \alpha_2}(\lambda)-\alpha_1$ for $n=1, 2, 3, \dots$ have irreducible factors of arbitrarily large degrees $k$ satisfying $\gcd(d, q^{k}-1)=1$. We state this prediction as a conjecture.

\begin{conjecture}
\label{conf:fpbar:perm-poly} Suppose $\alpha, \beta\in\mathbb{F}_q$. For $\l\in\Fpbar$, let $f_\l(z)\colonequals z^d+\l$. Define, as before, $P_{n, \alpha}(\lambda) = f_{\lambda}^{n}(\alpha)$; so, $P_{n, \alpha}(\lambda)\in \mathbb{F}_q[\lambda]$ is a polynomial of degree $d^{n-1}$. Suppose $\gcd(d, q-1)=1$. Then for each $M>0$, there is some $n\in\mathbb{N}$ such that the polynomial $P_{n, \alpha}(\lambda)-\beta$  
has an irreducible factor $g(\lambda)\in\mathbb{F}_q[\lambda]$  of degree $k>M$ satisfying $\gcd(d, q^{k}-1)=1$.
\end{conjecture}

The hypothesis $\gcd(d, q-1)=1$ is necessary. Indeed, if $\gcd(d, q-1)>1$, then $\gcd(d, q^k-1)>1$ for each $k\geq 1$. If $\gcd(d, q-1)=1$, then we can find infinitely many integers $k\geq 1$ such that $\gcd(d, q^k-1)=1$, so the conclusion of Conjecture~\ref{conf:fpbar:perm-poly} makes sense. Note that Conjecture~\ref{conf:fpbar:perm-poly} implies Conjecture~\ref{conf:fpbar} in the special case when $\alpha_1=\beta$.

%%%%%%%%%%%%%%%%%%%%%%%%%%%%%%%%%%%%%%%%%%%%%%%%%%%%%%%%%%%%%%%%%%%%%%%%%%%%%%%

\subsection{Computational evidence} 
\label{subsec:9.3}

We list several examples that provide evidence for both Conjecture~\ref{conf:fpbar} and Conjecture~\ref{conf:fpbar:perm-poly}. When counting irreducible polynomials over a finite field with specific properties, we always restrict our attention to monic irreducible polynomials.

\begin{example}
We work over $\mathbb{F}_2$. Let $d=3$, $\alpha_1=1$, $\alpha_2=0$, and $\beta=1$. We are in the special case when $\alpha_1=\beta$. Since $\gcd(3, 2^{k}-1)=1$ if and only if $k\in\mathbb{N}$ is odd, we seek any odd-degree irreducible factors of the polynomial $P_{n, 0}(\lambda)-1 = f_{\lambda}^{n}(0)-1$. Any root of such a factor will belong to $C(1, 0; 1)$. Table \ref{tab:F2_factors} lists the degrees of the irreducible factors of $P_{n, 0}(\lambda)-1$ for $1\leq n\leq 11$; newly appearing odd degrees are highlighted.

\begin{table}[h!]
\centering
\caption{Degrees of irreducible factors of $P_{n, 0}(\lambda)-1$ for $n=1, \dots, 11$}
\label{tab:F2_factors}
\begin{tabular}{@{}c|l@{}}
\hline\hline 
$n$ & Degrees of irreducible factors (all factors are simple) \\
\hline 
1   & \textbf{1} \\
2   & \textbf{3} \\
3   & 1, 3, \textbf{5} \\
4   & \textbf{27} \\
5   & 1, 38, 42 \\
6   & 18, \textbf{21}, \textbf{43}, \textbf{71}, 90 \\
7   & 1, 5, 38, \textbf{121}, 564 \\
8   & 3, \textbf{97}, 214, \textbf{375}, 1498 \\
9   & 1, 12, 16, \textbf{1205}, \textbf{5327} \\
10  & 3, \textbf{15}, 22, 22, 34, \textbf{61}, 82, \textbf{161}, 240, 334, 428, \textbf{4429}, 13852 \\
11 & 1, \textbf{16189}, \textbf{42859}  \\
\hline\hline 
\end{tabular}
\end{table}

The presence of factors with large odd degrees, such as $16189$ and $42859$  for $n=11$, strongly suggests that $C(1,0; 1)$ is infinite. Thus, Table~\ref{tab:F2_factors} numerically supports Conjecture~\ref{conf:fpbar:perm-poly}.

As direct factorization of $P_{n,0}(\lambda)-1$ (a polynomial of degree $3^{n-1}$) is computationally intensive, we also use the general algorithm. Table~\ref{tab:F2_gcd_factors} shows the number of irreducible polynomials $g(\lambda)$ of a given degree that divide $\gcd(P_{m,1}(\lambda)-1, P_{n, 0}(\lambda)-1)$ for some $m, n\in\mathbb{N}$.

\begin{table}[h!]
\centering
\caption{Irreducible polynomials that divide $\gcd(P_{m,1}(\lambda)-1, P_{n, 0}(\lambda)-1)$}
\label{tab:F2_gcd_factors}
\begin{tabular}{cccccccccccccc}
\hline\hline
degree of the irreducible polynomial & 1 & 2 & 3 & 4 & 5 & 6 & 7 & 8 & 9 & 10 & 11 & 12 & 13 \\
\hline
number of successful polynomials & 1 & 0 & 2 & 0 & 3 & 0 & 7 & 0 & 31 & 1 & 89 & 4 & 325 \\
\hline\hline
\end{tabular}
\end{table}
The final entry shows $325$ distinct irreducible polynomials of degree $13$ over $\mathbb{F}_2$. These alone yield $325\cdot 13 = 4225$ different values of $\lambda\in\overline{\mathbb{F}}_2$ in the set $C(1, 0; 1)$. Even without Table~\ref{tab:F2_factors}, we see that Table~\ref{tab:F2_gcd_factors} supports Conjecture~\ref{conf:fpbar}.
\end{example} 

\begin{example}
We work over $\mathbb{F}_5$. Let $d=3$, $\alpha_1 = 2$, $\alpha_2 = 1$, and $\beta = 2$. Again, we are in the special case $\alpha_1=\beta$. Next, $\gcd(3, 5^{k}-1)=1$ if and only if $k\in\mathbb{N}$ is odd. We seek any odd-degree irreducible factors of the polynomial $P_{n,1}(\lambda)-2 = f_{\lambda}^{n}(1)-2$. Table~\ref{tab:F5_d3_factors} lists the degrees of irreducible factors of $P_{n,1}(\lambda)-2$ for each $1\leq n\leq 12$; newly appearing odd degrees are highlighted. The presence of factors with large odd degrees, such as $163341$ for $n=12$, strongly suggests that $C(2,1; 2)$ is infinite. Thus, Table~\ref{tab:F5_d3_factors} numerically supports Conjecture~\ref{conf:fpbar:perm-poly}.

\begin{table}[h!]
\centering
\caption{Degrees of irreducible factors of $P_{n,1}(\lambda)-2$ for $n=1, \dots, 12$}
\label{tab:F5_d3_factors}
\begin{tabular}{@{}c|l@{}}
\hline\hline 
$n$ & Degrees of irreducible factors (all factors are simple) \\
\hline 
1   & \textbf{1} \\
2   & 1, 2 \\
3   & \textbf{3}, 6 \\
4   & 4, \textbf{7}, 16 \\
5   & 1, 1, \textbf{11}, 22, \textbf{45} \\
6   & 2, \textbf{5}, 10, 226 \\
7   & 26, 52, 250, \textbf{401} \\
8   & 1, 3, 3, 4, 5, 6, 13, 23, 64, \textbf{95}, \textbf{149}, \textbf{353}, 1468 \\
9   & 1, 20, \textbf{27}, \textbf{757}, 1082, 4674 \\
10  & 2, 4, \textbf{21}, 1632, 3932, 14092 \\
11  & 1, 6, \textbf{99}, 106, \textbf{205}, 280, 446, 778, 2370, 22642, 32116 \\
12  & 3, 3, 4, 7, 38, \textbf{71}, 13680, \textbf{163341} \\
\hline\hline 
\end{tabular}
\end{table}

\end{example}

\begin{example} We work over $\mathbb{F}_3$. Let $d=2$, $\alpha_1=0$, $\alpha_2=1$, and $\beta=2$. Table~\ref{tab:F3:d=2:[0,1,2]} shows the number of irreducible polynomials of a given degree dividing $\gcd(P_{m,0}(\lambda)-2, P_{n, 1}(\lambda)-2)$ for some $m, n\in\mathbb{N}$. The steadily growing counts suggest that $C(0, 1; 2)$ is infinite, thus supporting Conjecture~\ref{conf:fpbar}.

\begin{table}[h!]
\centering
\caption{Irreducible polynomials that divide $\gcd(P_{m,0}(\lambda)-2, P_{n, 1}(\lambda)-2)$}
\label{tab:F3:d=2:[0,1,2]}
\begin{tabular}{cccccccccccccc}
\hline\hline
degree of the irreducible polynomial & 1 & 2 & 3 & 4 & 5 & 6 & 7 & 8 & 9 & 10 & 11 & 12 & 13 \\
\hline
number of successful polynomials & 2 & 1 & 1 & 4 & 1 & 4 & 7 & 15 & 24 & 29 & 53 & 70 & 120 \\
\hline\hline
\end{tabular}
\end{table}
\end{example}

\begin{example} We work over $\mathbb{F}_9$ and let $\epsilon\in\mathbb{F}_{9}\setminus \mathbb{F}_3$ with $\epsilon^2=-1$. Let $d=2$, $\alpha_1=1$, $\alpha_2=\epsilon+1$, and $\beta=\epsilon$. The growing counts in Table~\ref{tab:F9:d=2:[1,e+1,e]} suggest that $C(1, \epsilon+1; \epsilon)$ is infinite, again supporting Conjecture~\ref{conf:fpbar:perm-poly}.

\begin{table}[h!]
\centering
\caption{Irreducible polynomials that divide $\gcd(P_{m,1}(\lambda)-\epsilon, P_{n, \epsilon+1}(\lambda)-\epsilon)$}
\label{tab:F9:d=2:[1,e+1,e]}
\begin{tabular}{ccccccc}
\hline\hline
degree of the irreducible polynomial & 1 & 2 & 3 & 4 & 5 & 6  \\
\hline
number of successful polynomials & 3 & 1 & 4 & 7 & 40 & 60 \\
\hline\hline
\end{tabular}
\end{table}
\end{example}

\begin{example}
 We work over $\mathbb{F}_8$ and let $\xi\in\mathbb{F}_{8}\setminus \mathbb{F}_2$ be an element that satisfies $\xi^3=\xi+1$. Let $d=3$, $\alpha_1=1$, $\alpha_2=\xi$, and $\beta=\xi^2+\xi+1$. Table~\ref{tab:F8} shows the count of irreducible polynomials of a given degree over $\mathbb{F}_8$ that divide $\gcd(P_{m,1}(\lambda)-(\xi^2+\xi+1), P_{n, \xi}(\lambda)-(\xi^2+\xi+1))$. The growing numbers suggest that $C(1, \xi; \xi^2+\xi+1)$ is infinite, which supports Conjecture~\ref{conf:fpbar}. This example also exhibits an interesting feature: there appear to be considerably more values of $\lambda$ whose minimal polynomial over $\mathbb{F}_8$ has an odd degree compared to an even degree. While we do not have a full explanation for this phenomenon, this parity imbalance also highlights the difficulty of Conjecture~\ref{conf:fpbar}. We note that for each $\l\in \mathbb{F}_{8^{2k+1}}$ (for each $k\in\N$), the polynomial $f_\l$ induces a permutation on $\mathbb{F}_{8^{2k+1}}$ and therefore, the problem of colliding orbits becomes a question of having \emph{one} periodic cycle (see also Subsection~\ref{subsec:9.2}) in $\mathbb{F}_{8^{2k+1}}$ (for $f_\l$) containing \emph{all three} points $\alpha_1,\alpha_2,\beta$. The numerical evidence from this example suggests that it is \emph{more likely} for the three points to belong to the same periodic cycle, rather than for there to be different preperiodicity portraits for the orbits of $\alpha_1$ and $\alpha_2$, both of which contain $\beta$.

 \begin{table}[h!]
\centering
\caption{Irreducible polynomials that divide $\gcd(P_{m,1}(\lambda)-(\xi^2+\xi+1), P_{n, \xi}(\lambda)-(\xi^2+\xi+1))$}
\label{tab:F8}
\begin{tabular}{ccccccc}
\hline\hline
degree of the irreducible polynomial & 1 & 2 & 3 & 4 & 5 & 6 \\
\hline
number of successful polynomials & 4 & 2 & 63 & 7 & 2265 & 31 \\
\hline\hline
\end{tabular}
\end{table}
\end{example}

\begin{example}
 We work over $\mathbb{F}_5$. Let $d=10$, $\alpha_1=1$, $\alpha_2=2$, and $\beta=3$. We consider this case because $p\mid d$, which one may suspect has a different answer; after all, the case $d=p^{\ell}$ \emph{does} exhibit a special behavior (see Theorem~\ref{thm:converse_ii}). Table~\ref{tab:F5:d=10:[1,2,3]} shows the count of irreducible polynomials of a given degree over $\mathbb{F}_5$ that divide $\gcd(P_{m,1}(\lambda)-3, P_{n, 2}(\lambda)-3)$. Once again, the growing numbers suggest that $C(1, 2; 3)$ is infinite, thus providing support for Conjecture~\ref{conf:fpbar}.

 \begin{table}[h!]
\centering
\caption{Irreducible polynomials that divide $\gcd(P_{m,1}(\lambda)-3, P_{n, 2}(\lambda)-3)$}
\label{tab:F5:d=10:[1,2,3]}
\begin{tabular}{cccccccccc}
\hline\hline
degree of the irreducible polynomial & 1 & 2 & 3 & 4 & 5 & 6 & 7 & 8 & 9 \\
\hline
number of successful polynomials & 1 & 1 & 1 & 2 & 8 & 13 & 18 & 43 & 103 \\
\hline\hline
\end{tabular}
\end{table}
\end{example}

%%%%%%%%%%%%%%%%%%%%%%%%%%%%%%%%%%%%%%%%%%%%%%%%%%%%%%%%%%%%%%%%%%%%%%%%%%%%%%%

\subsection{Numerical evidence for a more general question}
\label{subsec:9.4}

We switch now to a different family of polynomials:
\begin{equation}
\label{eq:g}
g_\l(z)=z^3+z+\l\text{ (parameterized by $\l\in\overline{\mathbb{F}_5}$);}
\end{equation}
also, we consider two starting points $\alpha_1,\alpha_2$ and one target point $\beta$. As before, we are interested in whether the set
\begin{equation}
\label{eq:C_g}
C_g(\alpha_1,\alpha_2;\beta)\colonequals \left\{\l\in\overline{\mathbb{F}_5}\colon \text{ there exist $m,n\in\N$ such that $g_\l^m(\alpha_1)=g_\l^n(\alpha_2)=\beta$}\right\}
\end{equation}
is infinite. For $\alpha_1=1$, $\alpha_2=3$ and $\beta=2$, we define the corresponding recurrence polynomials (for all $m,n\in\N$):
$$P_{g,m, 1}(\lambda)\colonequals g_{\lambda}^{m}(1)\text{ and }P_{g,n, 3}(\lambda)\colonequals g_{\lambda}^{n}(3).$$  
Table~\ref{tab:F5:d=3:[1,3,2]} shows the count of (monic) irreducible polynomials of a given degree over $\mathbb{F}_5$ that divides $\gcd\left(P_{g,m,1}(\lambda)-2, P_{g,n, 3}(\lambda)-2\right)$. The steadily growing numbers suggest that the set $C_g(1, 3; 2)\subseteq \overline{\mathbb{F}}_5$ corresponding to this polynomial $f_{\lambda}(z)=z^3+z+\lambda$ is infinite.

\begin{table}[h!]
\centering
\caption{Irreducible polynomials that divide $\gcd(P_{g,m,1}(\lambda)-2,
P_{g,n, 3}(\lambda)-2)$}
\label{tab:F5:d=3:[1,3,2]}
\begin{tabular}{cccccccccc}
\hline\hline
degree of the irreducible polynomial & 1 & 2 & 3 & 4 & 5 & 6 & 7 & 8 & 9 \\
\hline
number of successful polynomials & 1 & 0 & 1 & 5 & 6 & 17 & 24 & 32 & 114 \\
\hline\hline
\end{tabular}
\end{table}

This leads us to believe that the following Question has a positive answer.
\begin{question}
\label{quest:general}
Let $g\in\Fpbar[z]$ be a polynomial of degree $d\ge 2$, which is not an additive polynomial. We consider the family of polynomials $g_\l(z)=g(z)+\l$, parameterized by $\l\in\Fpbar$. Is it true that for each $\alpha_1,\alpha_2,\beta\in\Fpbar$, the set
\begin{equation}
\label{eq:C_g_2}
C_g(\alpha_1,\alpha_2;\beta)\colonequals \left\{\l\in\Fpbar\colon \text{ there exist $m,n\in\N$ such that $g_\l^m(\alpha_1)=g_\l^n(\alpha_2)=\beta$}\right\}
\end{equation}
is infinite?
\end{question}

A positive answer to Question~\ref{quest:general} supports the relevance of condition~\ref{conj:item-4} from Conjecture~\ref{conj:general}. 

%%%%%%%%%%%%%%%%%%%%%%%%%%%%%%%%%%%%%%%%%%%%%%%%%%%%%%%%%%%%%%%%%%%%%%%%%%%%%%%

\subsection{Collision of multiple orbits}
\label{subsec:9.5} 

We conclude the paper with a new question that leads us to an uncharted territory in the study of collision of orbits. Conjecture~\ref{conf:fpbar} considers the intersection of two orbits, which is the focus of this paper. This naturally leads to a more general question. So, given a polynomial $g\in\Fpbar[z]$ of degree $d$, which is not an additive polynomial, we let $g_\l(z)\colonequals g(z)+\l$ be a family of polynomials parameterized by $\l\in\Fpbar$. Then for any integer $s\ge 2$ and any given $\alpha_1, \dots, \alpha_s, \beta\in \Fpbar$, we define:
\begin{equation}
\label{eq:0010}
C_g(\alpha_1, \dots, \alpha_s; \beta) \colonequals \{\lambda\in\Fpbar: \text{ there exists } n_i\in \mathbb{N} \text { such that } f_\l^{n_i}(\alpha_i)=\beta \text { for } i=1,\dots, s\}.
\end{equation}
\begin{question} 
\label{quest:multiple}
Given a polynomial $g\in\Fpbar[z]$ of degree $d\ge 2$, which is not an additive polynomial, then using  the notation from \eqref{eq:0010}, what is the smallest integer $s\geq 2$ with the property that there exist
$\alpha_1,\dots,\alpha_s,\beta\in\Fpbar$ such that the corresponding set
$C_g(\alpha_1,\dots,\alpha_s; \beta)$ is finite?
\end{question}

We present one of the many examples we tested for our Question~\ref{quest:multiple}.
\begin{example}
 We work over $\mathbb{F}_5$, and let $\alpha_1=1$, $\alpha_2=2$, $\alpha_3=3$, $\beta=4$. We consider the two families of polynomials $f_{\lambda}(z)=z^3+\lambda$ and $g_{\lambda}(z)=z^4+z+\lambda$. 
 Tables~\ref{tab:multiple:orbits:x^3+lambda} and \ref{tab:multiple:orbits:x^3+x+lambda} show the counts of (monic) irreducible polynomials in both settings. The growing counts suggest that the sets $C_f(1, 2, 3; 4)$ and $C_g(1, 2, 3; 4)$ are both infinite.  We also see a qualitative difference: the counts for $C_f(1, 2, 3; 4)$ exhibit a parity imbalance (with a clear bias for odd degrees), while the counts for $C_g(1, 2, 3; 4)$ show a general upward trend as the degree increases. 

We believe that the imbalance from Table~\ref{tab:multiple:orbits:x^3+lambda} for odd degree polynomials may again be a consequence of the fact that for each  $\l\in \mathbb{F}_{5^{2k+1}}$, the polynomial $f_\l(z)$ induces a permutation polynomial on $\mathbb{F}_{5^{2k+1}}$. It seems far more likely that $\alpha_1,\alpha_2,\alpha_3,\beta$ all live in the same periodic cycle (under the action of $f_\l(z)$) rather than any other configuration of the preperiodicity portraits for the orbits of $\alpha_1,\alpha_2,\alpha_3$ intersecting at $\beta$. The latter scenario could happen when $\l\in \mathbb{F}_{5^{2k}}$ since the points may no longer be periodic under the action of $f_\l(z)$. 
\end{example}

\begin{table}[h!]
\centering
\caption{Irreducible polynomials  that divide $\gcd(P_{f, m,1}(\lambda)-4, P_{f, n, 2}(\lambda)-4, P_{f, k, 3}(\lambda)-4)$}
\label{tab:multiple:orbits:x^3+lambda}
\begin{tabular}{ccccccccc}
\hline\hline
degree of the irreducible polynomial & 1 & 2 & 3 & 4 & 5 & 6 & 7 & 8 \\
\hline
number of successful polynomials & 0 & 1 & 8 & 2 & 154 & 6 & 2732 & 28 \\
\hline\hline
\end{tabular}
\end{table}

\begin{table}[h!]
\centering
\caption{Irreducible polynomials that divide $\gcd(P_{g, m,1}(\lambda)-4, P_{g, n, 2}(\lambda)-4, P_{g, k, 3}(\lambda)-4)$}
\label{tab:multiple:orbits:x^3+x+lambda}
\begin{tabular}{ccccccccc}
\hline\hline
degree of the irreducible polynomial & 1 & 2 & 3 & 4 & 5 & 6 & 7 & 8 \\
\hline
number of successful polynomials & 1 & 0 & 3 & 5 & 7 & 10 & 24 & 43  \\
\hline\hline
\end{tabular}
\end{table}

Our numerical data suggest that the original intuition for collision of orbits must be revised when working over $\Fpbar$ due to the underlying finite combinatorics. More precisely, when the starting points $\alpha_j$ (for $j=1, \dots, s$), the target point $\beta$, and the parameter $\l$ all belong to a finite field $\mathbb{F}_{p^k}$, we are asking whether $\beta$ is contained in \emph{each} of the finite subsets $\OO_{f_\l}(\alpha_j)$ for $1\le j\le s$. Even though this is \emph{unlikely} for any \emph{given} $\l$, the positive answer becomes \emph{likely} once we look over \emph{all} $\l\in\mathbb{F}_{p^k}$. In contrast, when the field of definition for the starting points $\alpha_j$ and for the target point $\beta$ has positive transcendence degree, the collision of orbits is \emph{unlikely} without a well-defined global dynamical relation between the points (see Theorem~\ref{thm:main_0}). This new perspective in characteristic $p$ is reminiscent of the likely, unlikely, and impossible intersections studied in \cite{likely}, where the transcendence degree of the base field was also the crucial factor in determining the nature of an intersection.

%%%%%%%%%%%%%%%%%%%%%%%%%%%%%%%%%%%%%%%%%%%%%%%%%%%%%%%%%%%%%%%%%%%%%%%%%%%%%%%
%%%%%%%%%%%%%%%%%%%%%%%%%%%%%%%%%%%%%%%%%%%%%%%%%%%%%%%%%%%%%%%%%%%%%%%%%%%%%%%

\end{document}